\documentclass[12pt]{amsart}

\usepackage[all]{xy}
\usepackage{amssymb}
\usepackage{amsmath}
\usepackage[stable]{footmisc}

\usepackage[OT2,T1]{fontenc}
\DeclareSymbolFont{cyrletters}{OT2}{wncyr}{m}{n}
\DeclareMathSymbol{\Sha}{\mathalpha}{cyrletters}{"58}

\newcommand{\segment}[2]{\subsection{#2}\label{#1}}
\newcommand{\ssegment}[2]{\subsubsection{#2}\label{#1}}

\theoremstyle{definition}
\newtheorem*{prop*}{Proposition}

\theoremstyle{definition}
\newtheorem*{lemma*}{Lemma}

\theoremstyle{definition}
\newtheorem*{thm*}{Theorem}

\theoremstyle{definition}
\newtheorem*{cor*}{Corollary}

\theoremstyle{definition} 
\newtheorem{lemma}[subsubsection]{Lemma}
\newcommand{\Lemma}[2]{\begin{lemma} \label{#1} 
{#2} \end{lemma}}

\theoremstyle{definition} 
\newtheorem{prop}[subsubsection]{Proposition}
\newcommand{\Proposition}[2]{\begin{prop} \label{#1} 
{#2} \end{prop}}

\theoremstyle{definition} 
\newtheorem{thm}[subsubsection]{Theorem}
\newcommand{\Theorem}[2]{\begin{thm} \label{#1} 
{#2} \end{thm}}

\theoremstyle{definition} 
\newtheorem{conj}[subsubsection]{Conjecture}

\theoremstyle{definition}
\newtheorem*{conj*}{Conjecture}

\theoremstyle{definition} 
\newtheorem{cond}[subsubsection]{Condition}
\newcommand{\Condition}[2]{\begin{cond} \label{#1} 
{#2} \end{cond}}

\theoremstyle{definition} 
\newtheorem{rmk}[subsubsection]{Remark}

\newcommand{\disp}[3]{
\theoremstyle{definition}
\newtheorem{{#1}c}[subsection]{#1}
\begin{{#1}c} \label{#2} {#3} \end{{#1}c}
}

\newcommand{\sdisp}[3]{
\theoremstyle{definition}
\newtheorem{{#1}c}[subsubsection]{#1}
\begin{{#1}c} \label{#2} {#3} \end{{#1}c}
}

\let\oldmarginpar\marginpar
\renewcommand\marginpar[1]{\-\oldmarginpar[\raggedleft\footnotesize #1]%
{\raggedright\footnotesize #1}}

\newcommand{\opnm}{\operatorname}

\newcommand{\set}[2]{\big\{ #1 \; \big| \; #2 \big\} }

\newcommand{\Rep}{\operatorname{\bf{Rep}}}
\newcommand{\Vect}{\operatorname{\bf{Vect}}}
\newcommand{\LLie}{\operatorname{\bf Lie}}

\newcommand{\from}{\leftarrow}
\newcommand{\xto}{\xrightarrow}
\newcommand{\xfrom}{\xleftarrow}
\newcommand{\surj}{\twoheadrightarrow}

\newcommand{\gr}{\operatorname{gr}}

\renewcommand{\Im}{\operatorname{Im}}

\newcommand{\Spec}{\operatorname{Spec}}

\newcommand{\Spm}{\operatorname{Spm}}

\newcommand{\Li}{\operatorname{Li}}

\newcommand{\Hom}{\operatorname{Hom}}
\newcommand{\Ext}{\operatorname{Ext}}
\newcommand{\Isom}{\operatorname{Isom}}
\newcommand{\End}{\operatorname{End}}

\newcommand{\Aut}{\mathrm{Aut}\,}

\newcommand{\Lie}{\operatorname{Lie}}

\newcommand{\m}[1]{\mathrm{#1}}
\newcommand{\fk}[1]{\mathfrak{#1}}

\newcommand{\bb}[1]{\mathbb{#1}}
\newcommand{\cl}[1]{\mathcal{#1}}

\newcommand{\la}{\lambda}
\newcommand{\ka}{\kappa}
\newcommand{\si}{\sigma}
\newcommand{\Si}{\Sigma}
\newcommand{\ze}{\zeta}
\newcommand{\ga}{\gamma}
\newcommand{\al}{\alpha}
\newcommand{\be}{\beta}
\newcommand{\om}{\omega}
\newcommand{\Om}{\Omega}
\newcommand{\ep}{\epsilon}

\newcommand{\gl}{\operatorname{\fk{gl}}}

\newcommand{\Gm}{{\mathbb{G}_m}}
\newcommand{\Qp}{{\QQ_p}}
\newcommand{\Zp}{{\ZZ_p}}

\newcommand{\Cc}{\mathcal{C}}

\newcommand{\ZZ}{\bb{Z}}

\newcommand{\nN}{\fk{n}}
\newcommand{\gG}{\fk{g}}
\newcommand{\NN}{\bb{N}}
\newcommand{\QQ}{\bb{Q}}
\newcommand{\RR}{\bb{R}}
\newcommand{\PP}{\bb{P}}
\newcommand{\pP}{\fk{p}}

\newcommand{\Uu}{\mathcal{U}}
\newcommand{\Bb}{\mathcal{B}}

\renewcommand{\AA}{\bb{A}}

\newcommand{\Oo}{\mathcal{O}}
\newcommand{\Kk}{\mathcal{K}}
\newcommand{\Aa}{\mathcal{A}}
\newcommand{\Ee}{\mathcal{E}}
\newcommand{\Pp}{\mathcal{P}}
\newcommand{\Dd}{\mathcal{D}}
\newcommand{\qQ}{\fk{q}}
\newcommand{\Ff}{\mathcal{F}}
\newcommand{\Ll}{\mathcal{L}}
\newcommand{\Tt}{\mathcal{T}}

\newcommand{\Vv}{\cl{V}}
\newcommand{\Ww}{\cl{W}}

\newcommand{\inv}{^{-1}}

\newcommand{\areq}{\ar@{=}}
\newcommand{\suphook}{\ar@{^(->}}
\newcommand{\subhook}{\ar@{_(->}}

\newcommand{\smses}[6]
{
\[
\xymatrix{
1 \ar[r] &
#1 \ar[r]_-{#2} &
#3 \ar[r]_-{#4} &
#5 \ar[r] \ar@/_1.5pc/[l]_-{#6} &
1
}
\]
}

\newcommand{\inj}{\hookrightarrow}

\newcommand{\thrpl}{\PP^1 \setminus \{0,1,\infty\}}

\newcommand{\Fphi}{\boldsymbol {F\phi} }

\newcommand{\dR}{{\rm {dR}}}

\newcommand{\cris}{\m{cris}}
\newcommand{\oneato}{{{1_0}}}

\newcommand{\Words}{\operatorname{Words}}

\newcommand{\ab}{\mathrm{ab}}

\newcommand{\npr}{{n'}}

\newcommand{\Op}{\Oo_\pP}

\newcommand{\Zo}{{Z^o}}

\newcommand{\Eetil}{\widetilde \Ee}
\newcommand{\Aatil}{\widetilde \Aa}
\newcommand{\Ddtil}{\widetilde \Dd}

\newcommand{\ftil}{\widetilde f}
\newcommand{\PL}{\m{PL}}
\newcommand{\ev}{\fk{ev}}

\newcommand{\per}{\operatorname{per}}

\newcommand{\CUI}{\operatorname{FI}}
\newcommand{\Real}{\operatorname{Re}}

\newcommand{\IStdd}{\widetilde {^UI^{F\phi}_\m{Stdd}}}
\newcommand{\IAtil}{\widetilde{^UI_{\Aatil}}}
\newcommand{\atil}{\widetilde a}
\newcommand{\Ftil}{\widetilde F}

\newcommand{\MT}{\operatorname{\mathbf{MT}}}

\newcommand{\Res}{\operatorname{Res}}
\newcommand{\Etil}{\widetilde E}

\newcommand{\unVIC}{\operatorname{unVIC}}

\newcommand{\xyto}[2]{\underset{#2}{\overset{#1}\rightrightarrows}}

\newcommand{\Word}{\opnm{Word}}

\title{Mixed {T}ate motives and the unit equation II}

\author{Ishai Dan-Cohen}
\thanks{This work was supported by Priority Program 1489 of the Deutsche Forschungsgemeinschaft: {\it Experimental and algorithmic methods in algebra, geometry, and number theory}.} 

\date{\today}

\begin{document}

\begin{abstract}

Over the past fifteen years or so, Minhyong Kim has developed a framework for making effective use of the fundamental group to bound (or even compute) integral points on hyperbolic curves. This is the third installment in a series whose goal is to realize the potential effectivity of Kim's approach in the case of the thrice punctured line. As envisioned in \cite{mtmue}, we construct an algorithm whose output upon halting is provably the set of integral points, and whose halting would follow from certain natural conjectures. Our results go a long way towards achieving our goals over the rationals, while broaching the topic of higher number fields.

\bigskip

\end{abstract}

\maketitle

\setcounter{tocdepth}{1}
\tableofcontents

\raggedbottom
\SelectTips{cm}{11}

\section{Introduction}

\segment{int1}{}
This is the third installment in a series \cite{CKtwo,mtmue}\footnote{\emph{Explicit Chabauty-Kim theory for the thrice punctured line in depth two} = \emph{Mixed Tate motives and the unit equation 0}.}  devoted to what may reasonably be described as {\it explicit motivic Chabauty-Kim theory}. `Chabauty-Kim theory' refers to a framework developed by Minhyong Kim for making effective use of the fundamental group to bound, or conjecturally compute, integral solutions to hyperbolic equations. `Motivic' refers to the fact that while Kim's construction, in its original formulation, is $p$-adic \'etale, our methods are motivic. As things currently stand, this limits us to working in the mixed Tate, or Artin--Tate settings, that is, essentially to the projective line with (possibly interesting) punctures. So the adjective `motivic' implies a fairly specific context. While this context may seem narrow from a geometric point of view, it is quite broad from an arithmetic point of view, leading and relating to various interesting questions and conjectures.

`Explicit' refers to the fact that here our emphasis is on algorithms. {\it Explicit} Chabauty-Kim theory, as I see it, is somewhat orthogonal to Chabauty-Kim theory proper. If \textit{Chabauty-Kim theory} is about attempting to prove Kim's conjecture \cite{nabsd}, or at least about formulating and studying a range of related conjectures, \textit{Explicit} Chabauty-Kim theory is about making the theory \textit{explicit}. In particular, in the explicit theory, we allow ourselves to assume conjectures left and right, so long as those affect the halting, and not the construction, of the hoped-for algorithms. 

In this installment, we continue our study of $\thrpl$. We obtain an algorithm for computing the \emph{polylogarithmic Chabauty-Kim loci} (see below) over number fields which obey a certain technical condition. In turn, this technical condition is known for the rationals and follows for general number fields from a conjecture due to Jannsen. We also obtain an algorithmic solution to the unit equation over totally real fields obeying the same condition. Specializing to the case of the rationals, we obtain the algorithm envisioned in Dan-Cohen--Wewers \cite{mtmue}.

\segment{int2}{}
We now state our main application in more detail. Let $X = \thrpl$. Below, we construct an algorithm which takes as input an open integer scheme $Z$ (by which we mean an open subscheme of $\Spec \Oo_K$ for $K$ a number field), and outputs a subset of $X(Z)$.

\Theorem{mar1}{
Let $Z$ be a totally real open integer scheme. If our algorithm halts for the input $Z$, then its output is equal to the set $X(Z)$ of integral points of $X$ over $Z$.
}

\noindent
We also state / recall four conjectures: \emph{Zagier's conjecture}, \emph{Goncharov exhaustion (with weak control over ramification)}, the \emph{p-adic period conjecture}, and \emph{convergence of Chabauty-Kim loci for the polylogarithmic quotient}. Finally, we state our technical condition, which we call \emph{Hasse principle for finite cohomology}. We say that $K$ obeys \emph{Kim vs. Hasse} if convergence occurs {\it before} the Hasse principle fails (see below for details). The following proposition motivates the theorem above.

\Proposition{int3}{
Let $Z$ be a totally real open integer scheme with fraction field $K$.
\begin{enumerate}
\item
Assume Zagier's conjecture, Goncharov exhaustion, the $p$-adic period conjecture, and convergence of Chabauty-Kim loci for the polylogarithmic quotient hold for $Z$. Assume $K$ obeys Kim vs. Hasse. Then our algorithm halts for $Z$.
\item
Suppose $Z$ is contained in $\Spec \ZZ$. Assume Goncharov exhaustion, the $p$-adic period conjecture, and convergence of Chabauty-Kim loci for the polylogarithmic quotient hold for $Z$. Then our algorithm halts for $Z$. 
\end{enumerate}
}
\noindent
We refer to Theorem \ref{mar1} and Proposition \ref{int3}, taken together, as the ``equation-solving theorem''; see Theorem \ref{17s} for a more precise statement.

Practical (and unconditional) methods for solving the $S$-unit equation predate this work, and can be found, for instance, in de Weger \cite{deWeger} who uses the theory of logarithmic forms of Baker--W\"usthotz \cite{BakerWuestholz} (see also Everste--Gy\H{o}ry \cite{EversteGyory} for a general discussion). A more recent approach, due to K\"anel--Matchke \cite{KaenelMatschke} is based on the Shimura-Taniyama conjecture. Our primary purpose here is not to compete with these other methods, but rather, to develop Kim's theory in a special case, to explore its interaction with the theory of mixed Tate motives and motivic iterated integrals, and, in subsequent works, to provide new numerical evidence for Kim's conjecture. Also, while our focus here is \emph{explicit}, it is not yet \emph{practical}; see \S\ref{algprec} below.

\segment{}{}
We now give a brief indication of our main result, from which the equation-solving theorem follows as a corollary; precise statements, as well as more background, appear in section \ref{Thefirst} below. For this purpose, fix a prime $\pP \in Z$ which we assume to be totally split and recall that there is a commutative diagram like so. 
\[
\xymatrix{
X(Z) \ar[r] \ar[d]_-\ka &
X(Z_\pP) \ar[d]^-{\ka_\pP} \\
\Qp \otimes H^1(U^\PL_{\ge -n}) \ar[r]_-{\Re_\pP} &
H^1(U^{\PL, F\phi}_{\ge -n}). 
}
\]
Here $Z_\pP$ denotes the complete local scheme of $Z$ at $\pP$, isomorphic to $\Spec \Zp$, and $U^\PL_{\ge -n}$ denotes the level-$n$ quotient of the polylogarithmic quotient of the unipotent fundamental group of $X$ at the tangential base point $\vec{1_0}$, a certain quotient of a unipotent, motivic version of the fundamental group. The cohomology variety $H^1(U^\PL_{\ge -n})$ appearing below left is a certain $\QQ$-variety parametrizing torsors for $U^\PL_{\ge -n}$. The vertical map $\ka$ sends an integral point $x$ to the torsor of homotopy classes of paths 
\[
\oneato \to x.
\]
The cohomology variety appearing in the lower-right is a certain $p$-adic variant of the one to its left, based, as the notation suggests, on the theory of filtered $\phi$ modules. In terms of this diagram, we define
\[
X(Z_\pP)_n: = \ka_\pP\inv (\Im \Re_\pP).
\]
We construct an algorithm for computing the locus $X(Z_\pP)_n$ to given $p$-adic precision.

\begin{thm}[c.f. Theorem \ref{MainTh} below] \label{mar2}
Let $Z$ be a totally real open integer scheme, $\pP$ a totally split prime, $n$ a natural number, and $\ep >0$. If our algorithm halts for these inputs, then the functions $\widetilde F^\pP_i$ which the algorithm returns as output take values less than $\ep$ on $X(Z_\pP)_n$.
\end{thm}

\segment{}{}
The main problem of explicit Chabauty-Kim theory is to render the map $\Re_\pP$ computationally accessible; in the case at hand, we proceed as follows. Let $U(Z)$ denote the unipotent part of the fundamental group of the category of mixed Tate motives over $Z$. If $\Zo \subset Z$ is an open subscheme, there's an associated surjection
\[
U(\Zo) \surj U(Z).
\]
As part of the algorithm, we search for a $\Zo$ such that $U(\Zo)$ admits a {\it nice} set of coordinates. More will be said about the role played by $\Zo$ below; for now, let us fix $\Zo$ arbitrarily. A theory of $p$-adic iterated integration due to Coleman and Besser gives rise to a point
\[
I_{BC}: \Spec \Qp \to U(\Zo).
\] 
Our construction revolves around the following diagram.
\[
\xymatrix{
\Spec \Qp \times H^1(U^\PL_{\ge -n}) 
	\ar[r]^-{\Re_\pP} \ar@{=}[d] &
	H^1 (U^{\PL, F\phi}_{\ge -n})
	\ar@{=}[d] \\
\Spec \Qp \times Z^1(U(Z), U^\PL_{\ge -n})^\Gm
	\ar[r]^-{\ev_{I_{BC}}} \ar[d]_-{I_{BC} \times Id} &
	\Spec \Qp \times U^\PL_{\ge -n}
	\ar[d]^-{I_{BC} \times Id} \\
U(\Zo) \times Z^1(U(Z), U^\PL_{\ge -n})^\Gm
	\ar[r]_-{\ev_\m{Everywhere}} &
U(\Zo) \times U^\PL_{\ge -n}
}
\]
Here $Z^1(U(Z), U^\PL_{\ge -n})^\Gm$ denotes a certain space of $\Gm$-equivariant cocycles
\[
U(Z) \to U^\PL_{\ge -n},
\]
and the maps $\ev_{I_{BC}}$, $\ev_\m{Everywhere}$ are evaluation maps. Instead of attempting to compute the scheme-theoretic image $\Im \Re_\pP$ directly, we compute the pullback
\[
\tag{$*$}
(I_{BC} \times Id)\inv(\Im \ev_\m{Everywhere}).
\]

\segment{}{}
The group $U(\Zo)$ possesses certain special functions, known as \emph{motivic iterated integrals}, whose pullbacks along $I_{BC}$ are $p$-adic iterated integrals. The latter may be computed to arbitrary precision thanks to the algorithm of Chatzistamatiou--Dan-Cohen \cite{PItInts}, which we review in section \ref{Num} below. In order to compute $\Im \ev_\m{Everywhere}$ as well as its pullback ($*$) algorithmically, we need coordinates on $U(\Zo)$. Moreover, as the algorithm proceeds, we need to impose different, in fact contradictory, conditions on our coordinate system: to compute the pullback along $I_{BC}$, we need our coordinate functions to be given explicitly in terms of motivic iterated integrals. To compute the image
\[
\Im \ev_\m{Everywhere}
\]
however, we need coordinates compatible with the product on $U(\Zo)$. In the construction that follows, we attempt to bridge this gap; we fail in many ways, but are nevertheless able to make the error incurred arbitrarily small. 

This work does not have significant logical dependence on its predecessors \cite{CKtwo, mtmue}. The reason for this, in part, is that I found it preferable to modify portions of the work done in \cite{mtmue} in preparation for the construction of the algorithm. For instance, our use of the map ``$\ev_\m{Everywhere}$'' here (along with our acceptance of the $p$-adic period conjecture as yet another condition for halting) allows us to carry out the geometric part of the computation in a single step over the rationals and in a manner entirely divorced from arithmetic. In fact, the relationship between the present work and the latter is rather reversed: \cite{mtmue} may be seen as working out a particular example of the algorithm constructed here.

\segment{}{}
After making precise the conjectures and theorems indicated above in \S\ref{Thefirst}, we begin in segments \ref{-2}--\ref{-3} by studying formal properties of coordinate systems on $U(\Zo)$ which promise to shrink the apparent gap between computable properties of motivic iterated integrals and the desired compatibility with product. The result, which is summarized in propositions \ref{0} and \ref{.5}, consists of conditions on a basis $\Aa$ for the Hopf algebra $A(\Zo)$ of functions on $U(\Zo)$, given as a disjoint union of three subsets
\[
\Aa = \Ee \cup \Pp \cup \Dd,
\]
under which
\[
\Ee \cup \Pp
\]
forms an algebra basis of the polynomial algebra $A(\Zo)$, and the set $\Ee^\lor$ of dual elements forms a set of free generators for the Lie algebra
\[
\nN(\Zo) := \Lie U(\Zo).
\]
Our terminology gives a rough idea of the roles played by these subsets: $\Ee$ consists of {\it extensions}, $\Pp$ of {\it primitive non-extentions}, and $\Dd$ of {\it decomposables}.

\segment{}{}
Let $U(Z_\pP)$ (or $U(\Oo_\pP)$) denote the unipotent part of the fundamental group of the Tannakian category of mixed Tate filtered $\phi$ modules of Chatzistamatiou--\"Unver \cite{ChatUnv}.\footnote{
The same symbols might be used to denote the unipotent part of the fundamental group of the category of mixed Tate motives over $Z_p$ if such a category exists; but this hypothetical group will not intervene in this paper.
} Let $A(Z_\pP)$ denote the \emph{mixed Tate filtered $\phi$ Hopf algebra}, that is, the Hopf algebra of functions on $U(Z_\pP)$. Unlike the motivic Galois group $U(Z^\circ)$, the filtered $\phi$ Galois  group $U(Z_\pP)$ possesses a canonical set of free generators which give rise, dually, to a set of \emph{standard} basis elements in $A(Z_\pP)$. There is a \emph{realization map}
\[
\m{Re}_p: A(\Zo) \to \prod_{\pP | p} A(Z_\pP).
\]
Given a motivic iterated integral in $A(\Zo)$, we may wish to expand its realization in the standard basis. In segment \ref{4} we upgrade the algorithm of \cite{PItInts} to an algorithm which computes $p$-adic approximations of this expansion; we refer to this algorithm as the \emph{realization algorithm}. Examples of this algorithm are worked out in \cite[\S7.5]{mtmue} for $Z = \Spec \ZZ[1/2]$\footnote{
Actually, when $Z$ is an open subscheme of $\Spec \ZZ$ (and the higher motivic extension groups are hence of dimension $\le 1$), this algorithm may be replaced by a single direct period-computation; see Corwin--Dan-Cohen \cite{CorwinDCI, CorwinDCII}.
}\begin{enumerate}
\item
in segment 7.5.1 (when we compute, in the notation of that paper, the $p$-adic number $(\log^{F\phi} 2)(v_{-1})$), 
\item
in segment 7.5.2 (we compute, for instance, $(\log^{F\phi}2)^2(v_{-2})$), 
\item
and in segment 7.5.3 (computation of $\Li_3^{F\phi}(b)(w)$ for
\[
w \in \{v_{-1}^3, v_{-1}v_{-2}, v_{-2}v_{-1}, v_{-3}\}
).
\]
\end{enumerate}

\segment{}{}
Next comes our ``basis algorithm'' and our ``change of basis algorithm''. In comparison with the sketch of our algorithm in \cite[\S5.7]{mtmue}, the basis algorithm corresponds to step 1 (segment 5.7.1) while the change of basis algorithm is a part of step 2 (segment 5.7.2).\footnote{
Indeed, segment 5.7.2 of loc. cit. is a mixture of our \textit{change of basis} algorithm with our ``cocycle-image-evaluation algorithm''. 
} This material was inspired by Brown \cite{BrownDecomp}.

In segment \ref{6} below we attempt to construct a basis $\Aa$ of iterated integrals for $A(\Zo)$ (varying $\Zo \subset Z$ as we search) which fulfills the conditions of proposition \ref{0}. The result is our basis algorithm. Examples can be found in \cite[\S7.5]{mtmue} when we find a basis of $A(\Zo)_n$ for $n = 1,2,3,4$ consisting of unipotent motivic polylogarithms. For instance, in proposition 7.5.4.1 of loc. cit., we find that a basis for $A(\ZZ[1/2])_4$ is given by
\[
\Bb = \{(\log^U 2)^4, (\log^U 2)\ze^U(3), \Li^U_4(1/2) \}.
\]

When constructing the basis algorithm, one problem we face is that we are unable to verify algorithmically if a given iterated integral in $A(\Zo)_r$  belongs to the subspace
\[
E(\Zo)_r := \Ext^1_{\MT(Z)}(\QQ(0), \QQ(r)) \subset A(\Zo)_r
\]
of extensions. Using the realization algorithm, however, we can bound the distance between a given iterated integral and the extension space, and so bound the error thus incurred. We are thus forced to work with two potentially distinct bases. One basis, denoted by $\widetilde \Aa$, is given concretely and explicitly by motivic iterated integrals, but is imperfect in that its set
\[
\widetilde \Ee \subset \widetilde \Aa
\]
of alleged extensions may actually fail to be extensions. By projecting $\widetilde \Ee$ onto the space of extensions we obtain a second basis, $\Aa$, which is perfect in its fulfillment of the conditions of proposition \ref{0} on the one hand, but is merely \emph{abstract} on the other hand, as its definition is not constructive.

\segment{}{}
Let us briefly visit segment \ref{T5}, where the construction becomes somewhat intricate. The construction is recursive in $n \ge 2$. As soon as we have a basis $\widetilde \Aa_{\le n}$ of motivic iterated integrals in half-weights $\le n$, we want to also be able to expand an arbitrary iterated integral in half-weight $n$ in the given basis, or, more generally, to compute the inner product of two arbitrary iterated integrals in half-weight $n$ (for the standard inner product $\langle v_i, v_j \rangle = \delta_{i,j}$ induced by this basis).
 We don't hope to be able to do this precisely; instead we aim for an $\ep$-approximation
\[
\langle J, I \rangle_\ep.
\]
Segment \ref{T3} of our algorithm is key in setting the stage for this computation. Under our `Hasse principle', the realization map is injective near the extension groups (c.f. segment \ref{T8}). We may therefore require that the realization of our subset 
\[
\widetilde \Ee_n \subset \widetilde \Aa_n
\]
of near-extensions be linearly independent inside $\prod A(Z_\pP)$; the precise statement is complicated by the fact that our realization map is itself merely an approximation. Computation of the inner products
$
\langle J, I \rangle_\ep
$
for $J \in \Pp_n \cup \widetilde \Dd_n$ reduces to computations in lower weights via the \emph{Goncharov coproduct}, an explicit formula for the coproduct of two motivic iterated integrals due to Goncharov \cite{GonGal}. For the remaining inner products,
$
\langle J, I \rangle_\ep
$
with $J \in \widetilde \Ee_n$, we use the realization algorithm to map the remaining part $I'$ of $I$ and $J$ into $\prod A(Z_\pP)$ and compute there. Our requirement that $\m{Re}_p \; \widetilde \Ee_n$ be linearly independent ensures that the resulting system of linear equations will have a unique solution.

\segment{}{}
Given our abstract basis
\[
\Aa = \Ee \cup \Pp \cup \Dd
\] 
of $A(\Zo)$, we obtain a set
\[
\Si^o := \Ee^\lor
\]
of free generators for the Lie algebra $\nN(\Zo)$. The set $\Phi$ of words in $\Si^o$ forms a basis for the universal enveloping algebra $\Uu(\Zo)$; its dual
\[
\Ff \subset A(\Zo)
\]
gives us a new basis, which plays nicely with the Hopf-algebra structure; we call such a basis a \emph{shuffle basis}. We also have an exponential map
\[
\exp^\sharp: U(\Zo) \xto{\sim} S^\bullet \nN(\Zo)^\lor
\]
and we may compute the images
\[
\exp^\sharp(f_w)
\]
of the elements $f_w$ of $\Ff$ as Lie-words in $\Si^o$. We do not endeavor to carry this out explicitly here, contenting ourselves with the observation that the procedure is in an elementary sense algorithmic.

More interesting is the need to compare the two bases $\Aa$ and $\Ff$ of $A(\Zo)$. In segment \ref{8a}, we approximate such a comparison by using our imperfect, concrete basis $\widetilde \Aa$ to construct a near-shuffle basis $\widetilde \Ff$, and by computing the associated change-of-basis matrix to given $p$-adic precision. This is our change of basis algorithm. An example is worked out in segment 7.6.3 of \cite{mtmue}.

\segment{}{}
In terms of a set of generators $\Si^o$ for the unipotent motivic Galois group $U(\Zo)$, the problem of computing the image of $\ev_\m{Everywhere}$ becomes purely classical (if not quite formal); we make this precise in segment \ref{9a} below. We let $\nN^\PL$ denote the Lie algebra of the polylogarithmic quotient $U^\PL$ of the unipotent fundamental group of $X$, and we let $\nN(\Si^o)$ denote the free pronilpotent Lie algebra on $\Si^o$. The result is given as a finite family
\[
\{ F^\m{abs}_i \}_i
\]
of elements of 
\[
S^\bullet ( \nN^\PL_{\ge -n} \times \nN(\Si^o)_{\ge -n} )^\lor.
\]

\segment{}{}
The unipotent fundamental group $U(X)$ at the tangent vector $\oneato$ is canonically free prounipotent on two generators $e_0$, $e_1$ corresponding to monodromy about the punctures $0$ and $1$, respectively. As such, its coordinate ring $\Oo(U(X))$ possesses a canonical vector space basis 
$\{\Li_\om\}_\om$ 
where $\om$ ranges over words in the two generators. We abbreviate words in $e_0$, $e_1$ by words in $0$, $1$. The polylogarithmic quotient
\[
U(X) \surj U^\PL
\]
corresponds to the subalgebra generated by elements
\begin{align*}
\log &:= \Li_{0} \\  
\Li_1& := \Li_1, 
\quad \Li_2 := \Li_{10} 
\quad \Li_3 := \Li_{100},
\quad \dots
\end{align*}
In segment \ref{9b} we use the change-of-basis matrix of segment \ref{8a} to convert the elements $F_i^\m{abs}$ into functions on  
\[
U(\Zo) \times U^\PL_{\ge -n}
\]
given as elements $\widetilde F_i$ of the polynomial ring
\[
\QQ[\widetilde \Ee \cup  \Pp, \log, \Li_1, \dots, \Li_n].
\]
There is a natural map
\[
\QQ[\widetilde \Ee \cup  \Pp, \log, \Li_1, \dots, \Li_n] \to \m{Col}(X(Z_\pP))
\]
to the ring of Coleman functions, which we use to obtain the hoped-for family
\[
\{\widetilde F^\pP_i\}_i
\]
of Coleman functions. Having completed the construction of the ``Chabauty-Kim-loci'' algorithm, $\Aa_\m{Loci}$, we prove our main theorem in segment \ref{10a}.

\segment{}{}
In section \ref{Newt} we use Newton polygons to bound the number of roots in small neighborhoods, and in section 6 we review unpublished joint work with Andre Chatzistamatiou devoted to the computation of $p$-adic iterated integrals, on which we've already relied at several points.

We're now ready to construct the equation-solving algorithm of theorem \ref{mar1}. Joint work with David Corwin \cite{CorwinDCI} demonstrates the need for symmetrization with respect to the $S_3$ action on $X$. We set
\[
X(Z_\pP)_n^{S_3}:= \bigcap_{\si \in S_3} 
\si \big(
X(Z_\pP)_n
\big)
\]
(see \S\ref{int7} for the precise definition). We search for points to obtain a gradually increasing list
\[
X(Z)_n \subset X(Z).
\]
At the same time we construct Coleman functions $\widetilde F^\pP_i$ vanishing on $X(Z_\pP)_n^{S_3}$ and use those to obtain a gradually decreasing union of neighborhoods. Thus, roughly speaking, $X(Z)$ is sandwiched 
\[
X(Z)_n \subset X(Z) \subset X(Z_\pP)_n^{S_3}
\]
with $X(Z)_n$ gradually increasing while $X(Z_\pP)_n^{S_3}$ gradually decreases. We stop when the two sides meet. This concludes the construction of our equation-solving algorithm $\Aa_\m{ES}$, and allows us to state and prove the equation-solving theorem (segment \ref{17r}).

\segment{}{}
In section \ref{Beyond} we generalize theorem \ref{mar2} to allow arbitrary open integer schemes. Essentially the only difference is that one is forced to replace $X(Z_\pP)$ with the product $\prod_{\pP |p} X(Z_\pP)$. We are unable at this point, however, to obtain an equation-solving algorithm in this generality: to do so we would have to study solutions to systems of locally analytic functions on higher-dimensional spaces.

\segment{}{Near-term goals}

\ssegment{algprec}{Algorithmic precision}
The algorithmic constructions we make in this installment are precise by mathematical standards, but not by algorithmic standards, which are far more stringent. For instance, we do not attempt to make our $\ep$'s precise: any function of $\ep$ which is algorithmically computable, and which goes to zero with $\ep$, is again denoted by $\ep$ --- we refer to this as an \emph{admissible change in $\ep$}. Such imprecision is common in pure math, but useless for applications. Before going to Sage, we will of course have to compute exact levels of accumulated error as we make our approximations.

We also make no attempt to make our algorithm efficient: whenever we have a countable set, we don't hesitate to search through it arbitrarily. In fact, the problem of making our algorithm efficient interacts with significant, interesting problems of pure math; these include formulating explicit forms of Zagier's conjecture (available so far in only very special cases), and more precise versions of Goncharov's conjecture, at least with respect to ramification. Avoiding redundancy in our search through the set of iterated integrals is another interesting problem. Indeed, as Francis Brown has pointed out, as long as we limit ourselves to working with the polylogarithmic quotient, constructing a basis for all of $A(Z)_{\le n}$ is huge overkill: any $\Gm$-equivariant map
\[
U(Z) \to U(X)^\PL_{\ge -n}
\]
must factor through a small quotient of $U(Z)$ (easily computed in terms of abstract coordinates). A careful study of this quotient should yield a conjecture which is both much weaker than Goncharov's conjecture, and much more efficient for us.\footnote{
We alert the reader to the forthcoming works \textit{The Goncharov quotient in computational Chabauty-Kim theory I and II} by the author and David Corwin in which this is carried out for open subschemes of $\Spec \ZZ$ and new numerical results are obtained.
}

With regard to the endeavor to produce actual code, we would expect to push the computational boundary gradually, starting with very special cases in which the conjectures of Zagier and Goncharov are relatively well understood.

\ssegment{}{Comparison with Brown's method}
In the spring of 2014 I visited Francis Brown at the IHES in order to discuss the previous installment in this series \cite{mtmue}. Our meeting was inspiring and reassuring and helped me in developing the algorithm presented below.

Since then, Brown has made his own contribution to the subject \cite{BrownUnit} in which he develops a method (or a kind of blueprint) for constructing {\it many} polylogarithmic functions on the $p$-adic points of $\thrpl$ over an open subscheme of $\Spec \ZZ$ which vanish on integral points, at least when there are {\it enough} integral points. His idea is that if Goncharov's conjecture is false, there should actually be {\it more} such functions available, increasing the chances of isolating the set of integral points. His assumption that $X$ has many points replaces our reliance on Goncharov's conjecture for halting, and one of his goals, which he achieves in several examples, is to circumvent our construction of a basis of the mixed Tate Hopf algebra $A(Z)$. A particularly satisfying outcome is a more economic and aesthetic construction of the polylogarithmic function constructed in \cite{mtmue}.\footnote{
I should point out, however, that the purpose of the present work is somewhat different from Brown's work. Our goal here is to construct an actual algorithm, complete with precise criteria for halting. Moreover, its output should consist of functions that vanish not only on integral points, but also on the (a priori larger) loci defined by Kim's general theory. This ensures that the result may be used to produce numerical evidence for Kim's conjecture on the one hand, and provides an example of how Kim's framework for studying integral points may be made fully explicit and algorithmic on the other hand.}

Work remains to be done in comparing Brown's construction with ours, and hopefully, in harnessing the power of both approaches to construct more efficient, and more enlightening, algorithms.
\ssegment{}{Beyond totally real fields}
Most of the work completed here applies to arbitrary open integer schemes which obey {\it Kim vs. Hasse}. For the final application, however, we are limited to the totally real case. As mentioned above, in order to go further, we would have to develop methods for computing solutions to systems of polylogarithmic functions in higher dimensions.

\ssegment{}{Beyond the polylogarithmic quotient}
Beyond the polylogarithmic quotient, the motivic Selmer variety
\[
H^1(G(Z), U(X)_{\ge-n})
\]
is still canonically isomorphic to the space
\[
Z^1(U(Z), U(X)_{\ge-n})^\Gm
\]
of $\Gm$-equivariant cocycles. Explicit computation of this space is complicated however by the fact that the action of $U(Z)$ on $U(X)$ is highly nontrivial. This task is urgent for at least two reasons. Obviously, replacing the polylogarithmic quotient with the full unipotent fundamental group in our current algorithm would enable us to weaken our version of Kim's conjecture. More interesting, perhaps, would be the possibility of going beyond our three punctures $\{0,1,\infty\}$ to more general punctures, including possibly punctures that are not rational over the base-field in the context of mixed Artin-Tate motives.

\segment{ak}{Acknowledgements}
I would like to thank Stefan Wewers for helpful conversations during the conference on multiple zeta values in Madrid in December of 2014. I would like to thank Minhyong Kim, Amnon Besser, Francis Brown, Francesc Fit\'e, Go Yamashita for helpful conversations and email exchanges. I would like to thank Cl\'ement Dupont for long conversations during our time in Sarrians, and for pointing out a very helpful counterexample (see the appendix). I would like to thank Rodolfo Venerucci for conversations about finite cohomology. I would like to thank Jochen Heinloth and Giuseppe Ancona for help improving my presentation of the results. I would like to thank David Corwin for a careful reading and many helpful comments; moreover, in the course of our joint work \cite{CorwinDCI}, we discovered that conjecture \ref{int8} was false as stated in a previous draft. Finally, I wish to thank the referees for their helpful comments and suggestions.

\section{Conjectures and theorems}
\label{Thefirst}

We begin with a brief review of background material (unipotent fundamental group, Kim's conjecture, motivic iterated integrals, filtered $\phi$ iterated integrals, $p$-adic periods). For a more detailed exposition, tailored specifically to our applications, we refer the reader to Dan-Cohen--Wewers \cite{mtmue}.

\segment{int4}{A motivic variant of Kim's construction}

\ssegment{int5}{}
The prounipotent completion of the fundamental group of $X$ has a motivic precursor, known as the \emph{unipotent fundamental group}, a unipotent group object of the category of mixed Tate motives constructed by Deligne--Goncharov \cite{DelGon} whose various realizations had previously been studied by Deligne \cite{Deligne89}. We use the tangent vector $\oneato$ as base-point, and denote the resulting group simply by $U(X)$; see \S15 of Deligne \cite{Deligne89} for the use of tangential base-points in the various realizations. The unipotent fundamental group of $\Gm$ is equal to (the covariant total space of) $\QQ(1)$. The natural inclusion
\[
\thrpl \inj \Gm
\]
induces a surjection of unipotent fundamental groups
\[
U(X) \surj \QQ(1).
\]
Let $N$ denote its kernel. Then according to Deligne \cite[\S16]{Deligne89}, the Lie algebra of
\[
U(X)^\PL := U(X) / [N,N]
\]
is canonically a semidirect product
\[
\nN(X)^\PL = \QQ(1) \ltimes \prod_{i=1}^\infty \QQ(i).
\]
We write $\nN(X)^\PL_{\ge -n}$ for the quotient
\[
\QQ(1) \ltimes \prod_{i=1}^n \QQ(i),
\]
of $\nN(X)^\PL$, and $U^\PL_{\ge -n}$ for the corresponding quotient of $U(X)$. There are also associated quotients
\(
({_bP_\oneato})^\PL_{\ge -n}
\)
of the path torsors $_bP_{\oneato}$
 obtained by pushing out along the quotient map of
\[
U \surj U^\PL_{\ge -n}.
\] 
All this holds over $\Spec \ZZ$.

\ssegment{int6}{}
Let $Z \subset \Spec \Oo_K$ be an open integer scheme. Sending a point $b\in X(Z)$ to the torsor of {\it polylogarithmic paths} $({_bP_\oneato})^\PL_{\ge -n}$ defines a map
\[
\ka: X(Z) \to H^1(U^\PL_{\ge -n})(\QQ)
\]
to the set of rational points of a finite-type affine $\QQ$-scheme $H^1(U^\PL_{\ge -n})$ which parametrizes such torsors, the \emph{polylogarithmic Selmer variety}; see \S\ref{bcb1} below for a precise definition.

There is also a local $p$-adic version. We fix a closed point $\pP$ of $Z$ which we assume to be totally split for simplicity. We write $\Op$ instead of $\ZZ_p$ when we wish to emphasize the $\Oo_Z$-algebra structure, and similarly for $K_\pP = \Qp$. We denote $\Spec \Op$ by $Z_\pP$. We obtain a map
\[
\ka_\pP: X(\Op) \to H^1(U_n^{\PL, F\phi})(\Qp)
\]
to the set of $\Qp$-points of the \emph{filtered-$\phi$ polylogarithmic Selmer variety} $H^1(U_n^{\PL, F\phi})$. The set $X(\Op)$ may be viewed as the set of $\Qp$-points of the rigid analytic space (for instance) obtained from $\PP^1_\Qp$ by removing the residue disks about $0$, $1$, and $\infty$. Viewed this way, the map $\ka_\pP$ is locally analytic but not rigid analytic; rather, it is given in coordinates by Coleman functions. 

Connecting the filtered-$\phi$ and motivic versions is a map of $\Qp$-schemes
\[
\Re_\pP:
\Qp \otimes H^1(U^\PL_{\ge -n})
\to
H^1(U^{\PL, F\phi}_{\ge -n})
\]
induced by $p$-adic de Rham realization, which forms a commuting square
\[
\xymatrix{
X(Z) \ar[r] \ar[d] &
X(\Op) \ar[d]^{\ka_\pP} \\
 H^1(U^\PL_{\ge -n})(\Qp) \ar[r]_-{\Re_\pP} &
H^1(U^{\PL, F\phi}_{\ge -n})(\Qp). 
}
\]

\ssegment{int7}{}
The following conjecture was formulated jointly with David Corwin in \cite{CorwinDCI}. Let $\m{Col}(X(\Op))$ denote the ring of Coleman functions and consider the induced maps of $\Qp$-algebras
\[
\xymatrix{
 &
\m{Col} \big(X(\Op) \big)
 \ar@{<-}[d]^{\ka^\sharp_\pP}
 \\
\Qp \otimes \Oo 
\big(
 H^1(U^\PL_{\ge -n})
 \big)
  \ar@{<-}[r]_-{\Re^\sharp_\pP} 
 &
\Oo
\big(
H^1(U^{\PL, F\phi}_{\ge -n})
\big). 
}
\]
There is a natural action of the symmetric group $S_3$ on $\thrpl$, hence also on $\m{Col} \big(X(\Op) \big)$. We let 
$\ka_\pP^\sharp (\ker \Re_\pP^\sharp)$ 
denote the ideal of $\m{Col} \big(X(\Op) \big)$ generated by the image of $\ker \Re_\pP^\sharp$ and we let 
\[
\ka_\pP^\sharp (\ker \Re_\pP^\sharp) S_3
\]
denote the ideal generated by its orbit. In down to earth terms, this means that we close the set of generators $\{F_i(z)\}_i$ of the smaller ideal under the two operations 
\begin{align*}
\tag{*}
F_i \mapsto
F_i(1-z),
&&
\mbox{and}
&&
F_i \mapsto
F_i \left( \frac{1}{z} \right).
\end{align*}

We first define the \emph{polylgarithmic Chabauty-Kim locus at level $n$}
\[
X(\Op)_n \subset X(\Op)
\]
to be the vanishing locus (not a priori reduced) of the smaller ideal $\ka_\pP^\sharp (\ker \Re_\pP^\sharp)$. The polylogarithmic Chabauty-Kim loci form a nested sequence
\[
X(\Op) \supset X(\Op)_1 \supset X(\Op)_2 
\supset \cdots \supset X(Z).
\]
We note the following theorem, which is a direct consequence of the results of Kim \cite{KimTangential} via Soul\'e's \'etale regulator isomorphism \cite{SouleHigher}.

\begin{thm*}[Kim]
Suppose $Z$ is a totally real open integer scheme and let $\pP \in Z$ be any prime. Then for $n$ sufficiently large, $\Im \Re_\pP$ is contained in a subscheme of $H^1(U^{\PL, F\phi}_{\ge -n})$ of strictly lower dimension.
\end{thm*}

\noindent
Recall from Kim \cite{kimii} that the map $\ka_\pP$ has dense image. So as soon as we have a nonzero function on $H^1(U^{\PL, F\phi}_{\ge -n})$ vanishing on $\Im \Re_\pP$, the associated locus will be finite:

\begin{cor*}[Kim]
Suppose $Z$ is a totally real open integer scheme, and assume $\pP \in Z$ is totally split. Then for $n$ sufficiently large, the associated polylogarithmic Chabauty-Kim locus $X(\Op)$ is finite.  
\end{cor*}

\noindent
We define the \emph{symmetrized polylgarithmic Chabauty-Kim locus at level $n$}
\[
X(\Op)_n^{S_3} \subset X(\Op)
\]
to be the vanishing locus (not a priori reduced) of the ideal $\ka_\pP^\sharp (\ker \Re_\pP^\sharp)S_3$. Stretching Kim's conjecture from \cite{nabsd} somewhat, we propose the following.

\begin{conj} 
\textit{Convergence of polylogarithmic loci, joint with David Corwin.}
\label{int8}
Let $Z$ be a totally real open integer scheme, and $\pP \in Z$ a totally split prime. View $X(Z)$ as a locally analytic space over $\Qp$ with reduced structure. Then for $n$ sufficiently large, the associated polylogarithmic Chabauty-Kim locus satisfies
\[
X(\Op)_n^{S_3} = X(Z).
\]
\end{conj}

\begin{rmk}
The generalization from the rational to the totally real case should be harmless. By restricting attention to the polylogarithmic quotient, however, we are relying on a proper strengthening of Kim's conjecture. Nevertheless, since the codimension of $\Im \Re_\pP$ goes to infinity already for the polylogarithmic quotient, much of the motivation for Kim's conjecture does hold for the polylogarithmic quotient; see \cite{CorwinDCI} for a discussion of the role played by the $S_3$-orbit. Finally, our interpretation of $X(\Oo_\pP)_n$ as a potentially nonreduced space, implying that there should be no double roots for $n$ sufficiently large, is not discussed explicitly in \cite{nabsd}. 
\end{rmk}

\segment{int10}{Iterated integrals, $p$-adic periods, statement of arithmetic conjectures}

\ssegment{int11}{}
We begin by reviewing those properties of mixed Tate motives that we use. This material may be found, for instance, in Deligne-Goncharov \cite{DelGon}. We continue to work with an open integer scheme $Z$. Let $\MT(Z)$ denote the category of (unramified) mixed Tate motives over $Z$ with $\QQ$-coefficients. The category $\MT(Z)$ is $\QQ$-Tannakian. It has a special object $\QQ(1)$ of rank $1$. Every simple object is isomorphic to a unique $\QQ(n) := \QQ(1)^{\otimes n}$. Each object is equipped with an increasing filtration $W$, \emph{the Weight filtration}. (Since all the weights that occur are even, we also work with \emph{half-weights}, which are half the usual weights. These will be denoted by subscripts.) The functor
\[
\MT(Z) \to \Vect(\QQ)
\]
sending 
\[
E \mapsto \bigoplus \Hom(\QQ(i), \gr_{-2i}^W E)
\]
is a $\QQ$-valued fiber functor, with associated group of the form
\[
G(Z) = U(Z) \rtimes \Gm
\]
with $U(Z)$ free prounipotent. From generalities of mixed Tate categories, we have canonical isomorphisms
\[
U(Z)^\m{ab} = \bigoplus_{i\ge 1} \Ext^1_Z (\QQ(0), \QQ(n))^\lor.
\]
We have (highly nontrivial) canonical isomorphisms
\[
K_{2n-1}^{(n)}(Z) \xto{\sim} \Ext^1_Z(\QQ(0), \QQ(n)),
\]
and a computation of the dimensions of these $K$-groups via real-analytic methods due to Borel \cite{Borel53, Borel77}:
\[
\dim K_{2n-1}^{(n)}(Z) =
\left\{
\begin{matrix}
r_1+r_2 & \mbox{for } n \mbox{ odd } \ge 3 \\
r_2 & \mbox{for } n \mbox{ even } \ge 2,
\end{matrix}
\right.
\]
where $r_1$ (resp. $r_2$) denotes the number of real (resp. complex) places.

We let $\nN(Z)$ denote the Lie algebra of $U(Z)$, $\Uu(Z)$ its completed universal enveloping algebra, and $A(Z)$ the coordinate ring of $U(Z)$. Recall that the natural map
\[
\Uu(Z)^\lor \to A(Z)
\]
is an isomorphism of $\QQ$-vector spaces. 

\ssegment{int12}{}
Our discussion of iterated integrals applies to the complement in $\PP^1_Z$ of any divisor $\bf D$ which is a union of sections of 
\[
\PP_Z^1 \to Z
\]
and is \'etale over $Z$; we continue to use the letter $X$ which now denotes $\PP_Z^1 \setminus \bf D$ and we refer to the components of $\bf D$ as \emph{punctures}. We assume $\infty \in \bf D$. We say that a section of a vector bundle is \emph{nowhere vanishing} if its image in every closed fiber is nonzero. We define a \emph{base-point} to be either an integral point or a nowhere vanishing section of the normal bundle to one of the punctures. If $a$ is a base-point, we denote the unipotent fundamental group of $X$ at $a$ by $U_a(X)$. The unipotent fundamental group may be thought of as a prounipotent group object of $\MT(Z)$, or, after applying the canonical fiber functor, as a prounipotent $\QQ$-group equipped with an action of $G(Z)$, and we do not distinguish between these two points of view when we see no cause for confusion. 

If $b$ is a second base-point, we denote the unipotent path torsor by ${_bP_a}$; the latter may be thought of internally as a torsor-object of $\MT(Z)$ or externally as a $G(Z)$-equivariant torsor.

\ssegment{int13}{}
After forgetting the $G(Z)$-action, each unipotent path torsor ${_bP_a}$ is trivialized by a special $\QQ$-rational path
\[
_bp_a^\dR \in {_bP_a}(\QQ),
\]
and the fundamental group $U_a(X)$ is free on the set of logarithmic vector fields dual to the $1$-forms 
\[
\om_c = \frac{dt}{t-c}
\]
for $c$ a component of ${\bf D}_f:= {\bf D} \setminus \infty$; this is proved by Deligne \cite{Deligne89} when $K = \QQ$ and by Goncharov \cite{GonGal} in general.\footnote{
It is quite crucial that we work rationally here, since we will be using motivic iterated integrals to construct generators of the Hopf algebra $A(Z) = \Uu(Z)^\lor$ as a $\QQ$-algebra.
}
\footnote{
After tensorization with $K$, the canonical fiber functor on $\MT(Z)$ becomes canonically isomorphic to the de Rham fiber functor. A fact, which is perhaps underemphasized in the literature, however, is that the usual Tannakian interpretation of $U_a(X)$ in terms of unipotent connections is unavailable over the rationals unless $K = \QQ$. Thus, our ${_bp_a^\dR}$ is a \textit{rational form} of Deligne's canonical de Rham path.
} If $\om=(\om^1, \dots, \om^r)$ is a sequence of such differential forms, we let $f_\om$ denote the associated function (see section 2 of \cite{mtmue} for generalities on free prounipotent groups). 

We say that the datum $(a;\om;b)$ is \emph{combinatorially unramified} if the associated reduced divisors are \'etale over $Z$.\footnote{
If $a$ is a base-point, we let $a_0$ denote the {\it location} of $a$: $a_0=a$ if $a$ is an integral point, otherwise $a$ is a tangent vector {\it at} $a_0$. We assume that the datum $(a;\om; b)$ behaves nicely over $Z$ in an obvious sense, which requires several cases to state precisely: in all cases we assume the reduced divisor associated to $\om$, $a_0$, $b_0$, and $\infty$ is \'etale over $Z$; if $a$ is a tangent vector and $a_0 \neq b_0$, we assume the reduced divisor which supports $a+0+\infty$ on $\PP^1$ is \'etale over $Z$; if $a$, $b$ are both tangent vectors at the same point $a_0=b_0$, we assume similarly that the support of $a+b+0+\infty$ is \'etale over $Z$.
}  Being combinatorially unramified has the effect that the entire path bimodule $\Uu {_b P_a}$ (= bimodule over completed universal enveloping algebras at $a$ and $b$) is unramified over $Z$.

\ssegment{int13.6}{}
Following Goncharov \cite{GonGal}, we define $I_a^b(\om)$ to be the composite
\[
U(Z) \xto{ o( {_bp_a^\dR} ) } {_bP_a(X)} 
\xto{\sim} U_a(X) 
\xto{f_\om} \AA^1_\QQ.
\] 
Here $o({_bp_a^\dR})$ denotes the orbit map associated to the rational point ${_bp_a^\dR}$. If $A(Z)$ denotes the graded hopf algebra $\Oo(U(Z))$ then $I_a^b(\om)$ belongs to $A(Z)_r$. We refer to these elements as \emph{(combinatorially unramified) unipotent iterated integrals}. Among the unipotent iterated integrals are the \emph{classical unipotent polylogarithms}:
\[
\Li_{n+1}^U(t):= I_{\oneato}^t(0^{,n},1)
\]
(where the comma is used as a typographical pun to denote the concatenation product) and their single valued cousins $\Li_{n}^{U,sv}(t)$, for which we refer the reader to Brown \cite{BrownSingle}. In terms of these objects, we may state the conjectures of Zagier and Goncharov as follows. 

\begin{conj}[Zagier's conjecture]
\label{int14}
For each $n \ge 2$, the motivic Ext group
\[
E_n:= \Ext^1_K(\QQ(0), \QQ(n))
\]
is spanned by single valued unipotent $n$-logarithms $\Li_n^{U, sv}(t)$ with $t \in K$.
\end{conj}

\begin{rmk}
We recall that Zagier's conjecture is known for $n=2$ by Zagier \cite{Zagier?} and independently by work of Suslin \cite{Suslin?} and Bloch \cite{BlochDilogs}, for $n = 3$ by Goncharov, and for $K$ cyclotomic by Beilinson \cite{BeilinsonPolylog}. The main algorithm we construct below could be greatly simplified in the cyclotomic case $K=\QQ(\ze_N)$, where a basis for $E_n$ is given explicitly by the elements $\Li_n^{U,sv}(\ze_N^i)$ for $0 < i <N/2$. Conversely, away from the cyclotomic case, our algorithm is made complicated partly because of the lack of explicit constructions, even conjectural, of elements of $E_n$: away from the roots of unity, most $\Li_n^{U, sv}(t)$'s are {\it not} contained in $E_n$.
\end{rmk}

\begin{conj}[Goncharov-exhaustion]
\label{int15}
For any open integer scheme $Z$ and any $n\in \NN$, there is an open subscheme $\Zo \subset Z$ such that $A(\Zo)_{\le n}$ is spanned by combinatorially unramified unipotent iterated integrals over $\Zo$.
\end{conj}
\noindent

\begin{rmk}
Goncharov's conjecture \cite{GonICM} implies that for every $Z$ and $n$ there is an open subscheme $\Zo \subset Z$ such that $A(Z)_{\le n}$ (in place of $A(\Zo)_{\le n}$) is spanned by \emph{linear combinations of} combinatorially unramified unipotent iterated integrals over $\Zo$. This distinction between actual iterated integrals and linear combinations of iterated integrals is important, and our strengthening of the conjecture is, as far as I can tell, non-trivial and necessary for our purposes. The reason is that we are unable to check algorithmically if a given linear combination of combinatorially unramified iterated integrals in $A(\Zo)$ belongs to the subalgebra
\[
A(Z) \subset A(\Zo).
\]

Actually, our work here does provide an algorithm for checking if a given linear combination is $p$-adically close to $A(Z)$, but this algorithm is rather indirect. In outline, we first apply our \textit{basis algorithm} to obtain an open subscheme $\Zo \subset Z$ and a \textit{concrete} basis of $A(\Zo)$ (within the specified weight range) consisting of (actual!) combinatorially unramified unipotent iterated integrals compatible to within $\ep$ with the extension spaces. This gives rise to a set of generators of the Lie algebra, which in turn generate an \textit{abstract} shuffle basis of $A(\Zo)$. We then compare the two bases to within $\ep$ using our \textit{change of basis algorithm} and use the shuffle basis to identify $A(Z)$ inside $A(\Zo)$.

Thus, our conjecture represents a version of Goncharov's conjecture strengthened somewhat to include weak control over ramification. Better control over ramification would yield a faster algorithm. The case $Z = \Spec \ZZ \setminus \{2\}$, $n=\infty$ established by Deligne \cite{DelMuN}, and the discussion of the case $n=2$ in \cite{mtmue} both support the belief that unipotent iterated integrals should be compatible with ramification, at least to the extent predicted by our wording of the conjecture.

\end{rmk}

\ssegment{int16}{}
Unipotent iterated integrals have a filtered $\phi$ variant at each prime $\pP \in Z$. We mention only a few key similarities and differences, referring the reader to \cite[\S4]{mtmue} for details. As for mixed Tate motives, there is a Tannakian (in fact, mixed Tate) category of mixed Tate filtered $\phi$ modules, and an associated proalgebraic group of the form
\[
G(\Op) = \Gm \ltimes U(\Op)
\]
with $U(\Op)$ free prounipotent, but now over $\Qp$; we adopt our notation ($\nN(\Op)$, $\Uu(\Op)$, $A(\Op)$) from the motivic case.\footnote{Our use of $\Op$ (in place of $K_\pP$) in the notation expresses the fact that we're working with filtered $\phi$-modules as opposed to filtered $\phi, N$-modules.} Unlike the motivic case, $U(\Op)$ possesses canonical generators $v_{\pP,-1}, v_{\pP,-2}, v_{\pP,-3}, \dots$, and an associated special $K_\pP$-valued point
\[
u_\pP = \exp \sum_i v_{\pP,i}
\]
of $U(\Op)$.
 Note that the ``generators'' are not points of the group, but rather elements of the Lie algebra, regarded as Lie-like elements of the completed universal enveloping algebra --- see the discussion of free prounipotent groups in \S2 of \cite{mtmue}.

\ssegment{int17}{}
There is a morphism of unipotent groups 
\[
U(Z) \from U(Z_\pP)
\]
(linear over  $\Spec \QQ \from \Spec \Qp$) induced by filtered $\phi$ realization. The composite 
\[
U(Z) \from U(Z_\pP) \xfrom{u_\pP} \Spec K_\pP
\]
is the map denoted $I_\m{BC}$ above; we refer to it as ``Besser-Coleman integration''. The associated map of rings
\[
\per_\pP := I_\m{BC}^\sharp: A(\Op)\to K_\pP
\] 
is called the \emph{p-adic period map}. If $\pi$ denotes the embedding of $K$ in $K_\pP$, then we have
\[
I_a^b(\om)(I_\m{BC}) =  \per_\pP(I_a^b(\om)) = \int_{a^\pi}^{b^\pi}\om^\pi,
\]
a $p$-adic iterated integral in the sense of Coleman-Besser. (From this point of view, it's better to think of $I_a^b(\om)$ as a ``motivic iterated \emph{integrand}'': when we combine a {\it motivic iterated integrand} with {\it $p$-adic integration}, we obtain a {\it $p$-adic iterated integral}.) An algorithm for computing such integrals to arbitrary $p$-adic precision is constructed in Dan-Cohen--Chatzistamatiou \cite{PItInts}; as mentioned above, we review this unpublished work in section \ref{Num} below. The following conjecture is stated for instance in Yamashita \cite{YamashitaBounds}.

\begin{conj}[$p$-Adic period conjecture]
\label{int18}
Let $Z$ be an open integer scheme with fraction field $K$, and let $\pP$ be a closed point of $Z$. Then the $p$-adic period map
\[
\per_\pP: A(Z) \to K_\pP
\]
is injective.
\end{conj}

\ssegment{int19}{Hasse principle for finite cohomology}
In addition to the semilinear injectivity of the period conjecture, we will also need a linear injectivity property which concerns the product of realization maps
\[
\Re_p: \Qp \otimes \Ext^1_{\Oo_K}(\QQ(0), \QQ(n))
\to \prod_{\pP | p} \Ext^1_{\Op}(\Qp(0), \Qp(n))
\]
for $n \ge 2$. We recast this in the language of finite Galois cohomology as follows. Let $S$ be the set of places of $K$ above $p$ and $\infty$. Let $G_S$ denote the Galois group of the maximal extension of $K$ which is unramified outside of $S$, and for $v$ a place of $K$, let $G_v$ denote the total Galois group of the local field $K_v$. Following Bloch--Kato \cite{BlochKato}, we write  $H^i_{f}$ for the space of cohomology classes that are crystalline at all primes above $p$. By the $p$-adic regulator isomorphisms of Soul\'e \cite{SouleRegulators}:
\[
\Qp \otimes \Ext^1_{\Oo_K[1/p]}(\QQ(0), \QQ(n))
\xto{\sim} 
H^1(G_S, \Qp(n)),
\]
the injectivity of $\Re_p$ is equivalent to the following condition on a number field $K$, a prime $p$ of $\ZZ$ and an integer $n \ge 2$:

\Condition{hassest}{
The map
\[
\m{loc}_p : H^1_{ f}(G_S, \Qp(n)) \to 
\prod_{\pP | p} H^1_f (G_\pP, \Qp(n))
\]
is injective.}

Let $Z$ be a totally real open integer scheme with function field $K$ and assume the corresponding polylogarithmic Chabauty-Kim loci converge at $n$. We say that $Z$ \emph{obeys Kim vs. Hasse} if the above injectivity holds at levels $n' \le n$.

In fact, an anonymous referee has pointed out that \textit{condition \ref{hassest}} follows from a conjecture due to Jannsen. Conjecture 1 of \cite{Jannsen} (applied, in the notation of that article, to $X = \Spec \Oo_K [1/p]$) says that 
\[
H^2 (G_S, \Qp(n)) = 0 
\]
for $n < 0 $. By the Poitou-Tate exact sequence 
\[
\cdots \to
H^2 (G_S, M^\lor(1))^\lor \to
H^1(G_S, M) \to
\bigoplus_{v \in S} H^1(G_v, M) \to 
\cdots
\]
applied to $M = \Qp(n)$, we find that the map
\[
H^1(G_S, \Qp(n)) \to 
\bigoplus_{v \in S} H^1(G_v, \Qp(n))
\]
is injective whenever $n \ge 2$. 

For $v$ real and $p$ odd, we have
\[
H^1(G_v, \Qp(n)) = 0. 
\]
Indeed, if $C$ is a finite cyclic group with generator $\si$, and if we consider the elements $1-\si$, $N:= \sum_{\tau\in C} \tau$ of the group algebra $\ZZ[C]$, then for any $\ZZ[C]$-module $A$, the sequence
\[
0 \to A \xrightarrow{\si-1} A \xrightarrow{N} A \xrightarrow{\si-1} A \xrightarrow{N} \cdots \;,
\]
in which the first $A$ is in degree zero, forms a complex $A^\bullet$ and
\[
\m H^i (C,A)= \m H^iA^\bullet \;.
\]
When $C$ has order $2$, we have $N = \si +1$.  Applying this to our situation, we have 
\[
H^1(G_v, \mu_{p^r}^{\otimes n}) = 
H^1 \left( \ZZ/(2), (\ZZ/(p^r))^{\otimes n} \right)
\]
computed by the complex 
\[
0 \to \ZZ/(p^r) \xto{(-1)^n-1} \ZZ/(p^r) \xto{(-1)^n+1} 
\ZZ/(p^r) \to \cdots
\]
in which isomorphisms alternate with multiplication by $\pm 2$ depending on the parity of $n$. Either way, all cohomologies above degree $0$ vanish. 

Consequently, the map
\[
H^1(G_S, \Qp(n)) \to 
\bigoplus_{\pP | p} H^1(G_\pP, \Qp(n))
\]
is injective. Condition \ref{hassest} follows by restricting this map to finite cohomology spaces.

As a final remark, let us note that this injectivity is related to the non-vanishing of certain $p$-aidc $L$-values; c.f. theorem 4.2.1 of Perrin-Riou \cite{PRKubLeop}.

\segment{int20}{Outline of algorithm}

\ssegment{int20.1}{}
Our main construction is an algorithm, which we denote by $\Aa_\m{Loci}$, which takes as input an open integer scheme $Z$, a prime $p$ of $\ZZ$ over which $Z$ is totally split, a natural number $n$, and an $\ep$, and returns an open subscheme $\Zo$ of $Z$, an algebra basis $\widetilde \Bb$ of the polynomial ring $A(\Zo)_{\le n}$, and a family $\{\Ftil_i\}_i$ of elements of the polynomial ring 
\[
\QQ[\widetilde \Bb, \log, \Li_1, \dots, \Li_n].
\]

\ssegment{bcb1}{}
In terms of the category of mixed Tate motives $\MT(Z)$ and its Tannakian fundamental group $G(Z)$ discussed in  \S\ref{int11}, the polylogarithmic Selmer variety of \S\ref{int6} is characterized by the functor of $\QQ$-algebras
\[
R \mapsto H^1 \big( G(Z)_R, U^\PL_{\ge -n, R} \big).
\]
The proof of Proposition 2 of Kim \cite{kimi} applies mutatis mutandis to show that this functor is representable by a finite-type affine $\QQ$-scheme, which in this case is in fact isomorphic to affine space. 

\ssegment{int21}{}
According to \cite[\S5.2]{mtmue}, we have 
\begin{align*}
H^1(G(Z), U(X)^\PL_{\ge -n}) &=
Z^1(U(Z), U(X)^\PL_{\ge -n})^\Gm \\
	&=\Hom(U(Z), U(X)^\PL_{\ge -n})^\Gm,
\end{align*}
the space of $\Gm$-equivariant homomorphisms, and similarly for the filtered $\phi$ version over $Z_\pP$. Moreover, in the latter case, evaluation at $u_\pP$ induces an isomorphism
\[
\ev_{u_\pP}: \Hom(U(Z_\pP), U(X_\pP)^\PL_{\ge -n})^\Gm \xto{\sim}
U(X_\pP)^\PL_{\ge -n} = \Qp \otimes U(X)^\PL_{\ge -n}.
\]
The composit map
\begin{align*}
\Qp \otimes \Hom(U(Z), U(X)^\PL_{\ge -n})^\Gm
 \to
 \Hom(U(Z_\pP), &U(X_\pP)^\PL_{\ge -n})^\Gm 
 \\
 & \xto{\sim}
\Qp \otimes U(X)^\PL_{\ge -n}
\end{align*}
is given by evaluation at the pullback $I_\m{BC}$ of $u_\pP$ to $U(Z)$. As explained in the introduction, in order to compute its scheme-theoretic image, we first put this evaluation map inside the universal family of evaluation maps $\ev = \ev_\m{Everywhere}: $
\[
\Hom^\Gm(U(Z), U(X)^\PL_{\ge -n}) \times
U(\Zo) \to U(X)^\PL_{\ge -n} \times U(\Zo)
\]
pulled back along
\[
U(\Zo) \surj U(Z).
\]
If we fix arbitrary generators of $U(\Zo)$, these give rise to coordinates on $A(\Zo)$, which we refer to as {\it abstract shuffle-coordinates}. In terms of these, the computation is purely classical. We must then however switch to coordinates whose image under the period map can be computed, that is, to \emph{concrete coordinates} given by unipotent iterated integrals. As explained in the introduction, the heart of our algorithm constructs such coordinates, as well as an approximate change-of-basis matrix which relates a judicious choice of abstract shuffle-coordinates to our concrete coordinates. This key step is inspired by the work of Francis Brown in \cite{BrownDecomp}.

\segment{main}{Statement of main theorem}
For each prime $\pP$ lying above $p$, the $\pP$-adic period map extends in an obvious way to a map
\[
\QQ[\widetilde \Bb, \log, \Li_1, \dots, \Li_n] \to \m{Col}(X(\Op))
\]
to the ring of Coleman functions; denote the image of the element $\Ftil_i$ from segment \ref{int20.1} by $\Ftil_i^\pP$.

\begin{thm} \label{MainTh}
Let $Z$ be an open integer scheme, $\pP \in Z$ a totally split prime, $p$ the image of $\pP$ in $\Spec \ZZ$, $n$ a natural number, and $\ep \in p^\ZZ$. Let
\[
\Kk_\pP(\nN^\PL_{\ge -n}) \lhd \m{Col}(X(Z_\pP))
\]
denote the ideal which defines the Chabauty-Kim locus $X(Z_\pP)_n$; we refer to $\Kk_\pP(\nN^\PL_{\ge -n})$ as the \emph{$p$-adic Chabauty-Kim ideal associated to $\nN^\PL_{\ge -n}$}.
\begin{enumerate}
\item
Suppose $\Aa_\m{Loci}(Z,p,n,\ep)$ halts. Then there are functions $\{ F^\pP_i\}$ generating the $\pP$-adic Chabauty-Kim ideal $\Kk_\pP(\nN^\PL_{\ge -n})$ associated to $\nN^\PL_{\ge -n}$, such that 
\[
\left| \widetilde F_i^\pP - F_i^\pP \right| < \ep
\] 
for all $i$.
\item
Suppose Zagier's conjecture (conjecture \ref{int14}) holds for $K$ and $n' \le n$. Suppose Goncharov exhaustion (conjecture \ref{int15}) holds for $Z$ and $n' \le n$. Suppose the period conjecture holds for the open subscheme $\Zo \subset Z$ constructed in segment \ref{6} in half-weights $n' \le n$. Suppose $K$ obeys the Hasse principle for finite cohomology (condition \ref{hassest}) in half-weights $2 \le n' \le n$. Then the computation $\Aa_\m{Loci}(Z,p, n, \ep)$ halts.
\end{enumerate}
\end{thm}
\noindent
We remark that part (1) of the theorem is independent of the choice of norm on the space of polylogarithmic functions up to an admissible change in $\ep$. We complete the construction of the Loci algorithm and prove Theorem \ref{MainTh} in Section 4.

\section{Construction of arithmetic algorithms}

\segment{-2}{Generators for graded free algebras}

\Proposition{-1}
{
Let $S = \bigcup_{i=1}^\infty S_i$ be a disjoint union of finite sets, and similarly 
\[
S' = \bigcup_{i=1}^\infty S'_i.
\] 
Let $k$ be a field and $k[S], k[S']$ associated graded free algebras and $I,I'$ the augmentation ideals. Let 
\[
\phi: k[S'] \to k[S]
\]
be a homomorphism which preserves the grading. Suppose the induced map
\[
I'/I'^2 \to I/I^2
\]
is iso. Then $\phi$ is iso.
}

\begin{proof}
For $n\ge 1$, we have $I_n = k[S]_n$. Surjectivity follows by induction using the short exact sequences
\[
0 \to (I^2)_n \to k[S]_n \to (I/I^2)_n \to 0.
\] 

Since $S_i$ maps to a basis of $(I/I^2)_i$, the bijection
\[
(I'/I'^2)_i \to (I/I^2)_i
\]
gives us a bijection between $S'_i$ and $S_i$. For any $n$, $\phi$ maps $S'_{\le n}$ into $k[S]_{\le n}$, so $\phi$ restricts to a map
\[
k[S'_{\le n}] \to k[S_{\le n}]
\]
of subalgebras generated in degrees $\le n$. These are surjective maps of polynomial algebras of same finite Krull dimension. This means that 
\[
\Spec  k[S_{\le n}] \to \Spec k[S'_{\le n}]
\]
is a closed immersion between affine spaces of same dimension, hence an isomorphism by the Hauptidealsatz.
\end{proof}

\segment{-3}{Generators for mixed Tate groups}

\ssegment{mtg}{}
For a review of free prounipotent groups, we refer the reader to \S2 of Dan-Cohen--Wewers \cite{mtmue}. By a \emph{mixed Tate group} over a field $k$ of characteristic zero, we mean a free prounipotent group $U$ equipped with a grading of the Lie algebra
\[
\nN = \Lie U
\] 
such that $\nN_i = 0$ for $i \ge 0$. The Lie algebra $\nN$ admits a set of homogeneous free generators, and we define a \emph{set of homogeneous free generators of $U$} to be a set of homogeneous free generators of $\nN$. The grading on $\nN$ induces also a grading of the completed universal enveloping algebra $\Uu = \Uu \nN$ such that $\Uu_0 = k$ and $\Uu_i = 0$ for $i >0$, as well as a grading on the coordinate ring $A = \Oo(U) = \Uu^\lor$ such that $A_0 = k$ and $A_i = 0$ for $i<0$. We refer to the graded degree of an element (of $\nN$, $\Uu$, $A$) as its \emph{half-weight}.

The kernel of the comultiplication
\[
E_n < A_n
\]
is the space of \emph{extensions}. Indeed, by the general theory of mixed Tate categories we have exact sequences
\[
0 \to \Ext^1_{\Rep (\Gm \ltimes U)} \big(k(0), k(n) \big)
\to
A_n \to 
\underset{i,j \ge 1}{\bigoplus_{i+j=n}} A_i \otimes A_j
\]
where $k(i)$ denotes the trivial $U$-representation in half-weight $-i$. Similarly, the multiplication gives rise to a subspace
\[
A_n > D_n,
\]
namely the image of the map
\[
A_n \from 
\underset{i,j \ge 1}{\bigoplus_{i+j=n}} A_i \otimes A_j;
\]
we refer to $D_n$ as the \emph{space of decomposable elements}.

\Proposition{0}
{
Let $U$ be a mixed Tate group, and let $A$ denote its coordinate ring. For each $i$ let $E_i$ denote the space of extensions in $A_i$, $D_i$ the space of decomposable elements. Let $\Pp_i$ be a linearly independent subset of $A_i$ which spans a subspace $P_i$ complementary to $E_i + D_i$.  Let $\Ee_i$ be a basis for $E_i$ and let $ \Ee = \bigcup \Ee_i$, $\Pp = \bigcup \Pp_i$. Then as a ring,
\[
A = k[\Ee \cup \Pp].
\]
}

\begin{proof}
The subspaces $E_i$ and $D_i$ are disjoint. To see this, fix an arbitrary set
\[
\ep' = \bigcup_{i  = 1}^\infty \ep'_{-i}
\]
of homogeneous free generators for $U$, and for $w$ a word in $\ep'$, let $f_w \in A$ denote the associated function. Then $E_i$ has basis
\[
\set{f_a}{a \in \ep'_{-i}}
\]
dual to the set of one-letter words of half-weight $-i$, while $D_i$ is spanned by shuffle products of functions $f_w$ with $w$ a word in $\ep'_{>-i}$, so is contained in the space with basis
\[
\set{f_w}{w \in \Words_{-i}(\ep'_{>-i})}.
\]
It follows that $A_i$ decomposes as a direct sum
\[
A_i = E_i \oplus P_i \oplus D_i
\tag{\ref{0}$*$}
\]
and that $\Ee_i \cup \Pp_i$ maps to a basis of $(I/I^2)_i$. Hence, by Proposition \ref{-1}, $\Ee \cup \Pp$ forms a set of free $k$-algebra generators  for $A$. 
\end{proof}

\Proposition{.5}
{
In the situation and the notation of Proposition \ref{0}, let
\[
\ep_{-i} \subset \Uu_{-i}
\]
be the set of elements dual to the elements of $\Ee_i$ relative to the decomposition (\ref{0}$*$). Then
\[
\ep := \bigcup_{i=1}^\infty \ep_{-i}
\]
forms a set of free generators for $U$.
} 

\begin{proof}
We claim that every element
\[
\ep_{-i,j} \in \ep_{-i}
\]
is of Lie type: if 
\[
\nu: \Uu \to \Uu \otimes \Uu
\]
denotes the comultiplication, then
\[
\nu(\ep_{-i,j}) = 1\otimes \ep_{-i,j} + \ep_{-i,j} \otimes 1.
\tag{$*$}
\]
Let
\[
\Pp'_{i} = \Ee_{i} \cup \Pp_{i}.
\]
According to Proposition \ref{0}, the set $\Dd_i$ of monomials in $\Pp'_{< i}$ forms a basis of $D_i$. Let 
\[
\Aa_i = \Ee_i \cup \Pp_i \cup \Dd_i.
\]
It suffices to check the equality ($*$) after pairing with an arbitrary basis element
\[
\Aa_{i',j'}\otimes \Aa_{i'', j''} 
\]
of $A_{i'} \otimes A_{i''}$
with $ i'+i'' = i$. We have
\begin{align*}
\langle \nu(\ep_{-i,j}),  \Aa_{i',j'}\otimes \Aa_{i'', j''}  \rangle
	&= \langle \ep_{-i,j},  \Aa_{i',j'} \cdot \Aa_{i'', j''}  \rangle\\
	&= 	\left\{
		\begin{tabular}{ll}
		1 & {\mbox if  $\Aa_{i',j'} = 1$ and 
			$\Aa_{i'', j''} = \Ee_{i,j}$ is dual to $\ep_{-i, j}$ } \\
		1 & {\mbox if $\Aa_{i',j'} =  \Ee_{i,j}$ and $ \Aa_{i'', j''} = 1$ } \\
		0 & \mbox{ otherwise}
		\end{tabular}
		\right. \\
	&= \langle 1 \otimes \ep_{-i,j},  \Aa_{i',j'}\otimes \Aa_{i'', j''}  \rangle +
		\langle \ep_{-i,j} \otimes 1,   \Aa_{i',j'}\otimes \Aa_{i'', j''}  \rangle \\
	&= \langle 1 \otimes \ep_{-i,j} + \ep_{-i,j} \otimes 1,  \Aa_{i',j'}\otimes \Aa_{i'', j''}  \rangle
\end{align*}
which shows that $\ep_{-i,j}$ is of Lie type as claimed. 

It follows that $\ep_{-i}$ is a subset of the graded piece $\nN_{-i}$ of the Lie algebra $\nN \subset \Uu$, which maps to a basis of $\nN^\ab_{-i}$. It follows that $\ep$ forms a set of free generators as stated. 
\end{proof}

\segment{1}{}
Recall that by an \emph{open integer scheme} we mean on open subscheme
\[
Z \subset \Spec \Oo_K,
\]
$K$ a number field. By a \emph{number scheme} we mean $\Spec K$, $K$ a number field. Given $Z$ an open integer or number scheme, we let
\[
A(Z) = \Oo(U(Z))
\]
denote the graded Hopf algebra of unramified mixed Tate motives over $Z$. 

\segment{FI}{}
Given an open integer scheme $Z$ with function field $K$ and a unipotent iterated integral $I_a^b(c_1, \dots, c_r) \in A(K)_n$, we say that $I$ is \emph{combinatorially unramified over $Z$} if the associated reduced divisor
\[
D = \{a, b, c_1, \dots, c_r\}
\]
is \'etale over $Z$. We denote the $\QQ$-vector space of formal linear combinations of such tuples $(a; c_1, \dots, c_r; b)$ by $\CUI(Z)_r$, the space of \emph{formal integrands in half-weight $r$}.

\segment{2}{}
If $k$ is a field equipped with an absolute value $|\cdot|$, we say that a subset of $k^n$ is $\ep$-linearly independent if each of the associated determinants has absolute value greater than $\ep$.

\segment{3}{}
Let $Z$ be an open integer scheme and $p \in \ZZ$ a prime such that $Z$ is totally split above $p$. Recall that $A(\Oo_\pP)$ denotes the graded Hopf algebra of mixed Tate filtered $\phi$ modules over $K_\pP$, and recall that $A(\Oo_\pP)$ possesses a \emph{standard basis}. We say that a subset 
\[
\Pp \subset A(Z)_n
\]
is $\ep$-\textit{linearly independent relative to $\Re_p$} if its image in $\prod_{\pP | p}A(\Oo_\pP)_n$ is $\ep$-\textit{linearly independent} with respect to the standard basis.

\segment{4}{Realization algorithm}

\ssegment{13a}{}
We Recall from paragraph \ref{int16} that $U(\Zp)$ denotes the unipotent fundamental group of the category of mixed Tate filtered $\phi$ modules, that it contains a special $\Qp$-point $u$, and that the family
\[
v_i = (\log u)_i \in \nN(\Qp)
\]
for $i \in \ZZ_{\le -1}$ forms a set of free generators. The associated shuffle basis of $A(\Zp)$ (which is dual to the basis of the universal enveloping algebra consisting of words in the generators)
 is what we call the \emph{standard basis}. We now construct an algorithm for evaluating an iterated integral $I_a^b(\om)$, whose associated divisor $D$ is a union of $\Zp$-points, on a word
\[
w =v_{-i_r} \cdots  v_{-i_2}  v_{-i_1}
\]
in the generators $v_i$ to given precision $\ep$. 

\ssegment{bbb1}{}
Let $Z \subset \Spec \Oo_K$ be an open integer scheme, and $\pP \in Z$ a prime which is totally split. Recall from segment \ref{FI} that $\CUI(Z)_r$ denotes the $\QQ$-vector space of combinatorially unramified integrands in half-weight $r$. The \emph{realization algorithm}, alluded to above and constructed in segment \ref{bee3} below, may be interpreted as an algorithm which takes a natural number $r$ and an $\ep \in p^\ZZ$ as input, and returns a linear map
\[
\IStdd: \CUI(Z)_r \to \prod_{\pP|p} A(\Op)_r
\]
given explicitly by a matrix with rational entries. If
\[
\Real: A(Z)_r \to \prod_{\pP|p} A(\Op)_r
\]
denotes the realization map, and 
\[
{^UI}: \CUI(Z)_r \to A(Z)_r
\]
denotes the map taking an integrand to the associated unipotent iterated integral, then the triangle
\[
\xymatrix{
\CUI(Z)_r \ar[d]_-{^UI} \ar[dr]^-{\IStdd} \\
A(Z)_r  \ar[r]_-{\Real} &
\prod_{\pP|p} A(\Op)_r
}
\]
fails to commute by at most $\ep$. Said differently, $\IStdd$ is an approximation of the matrix representing the composite
\[
{^U I^{F\phi}} :=\Real \circ ({^U I})
\]
with respect to the `standard' bases on source and target. An example is worked out in segment 7.5.3 of \cite{mtmue}. Note, however, that the triviality of the motivic Galois action on the polylogarithmic quotient makes that example deceptively simple compared to the general algorithm that follows. 

We begin in segments \ref{ant1}--\ref{ant6} by deriving a formula (Lemma \ref{ant2}) for the action of the special $\Qp$-point $u$ of the unipotent mixed Tate filtered $\phi$ Galois group $U(\ZZ_p)$ on any generator of the unipotent fundamental group of $X := \AA^1_\Qp \setminus D$. We then note in segment \ref{ant7} that the formula of Lemma \ref{ant2} gives rise to an algorithm for computing the action of $u$ on any word in the generators. In segment \ref{ant8} we extend this algorithm to include not only fundamental groups but also path torsors. After a few elementary observations regarding the matrix entries of a graded representation of a graded Lie algebra on a graded vector space equipped with a graded basis on which we do not wish to impose an ordering (segments \ref{bee1}--\ref{bee2}), and after constructing a certain family of polynomials with rational coefficients based on these observations (segment \ref{bee2.5}), we construct the realization algorithm in segment \ref{bee3} and we state and verify the correctness of its output in segments \ref{bee4}--\ref{bee5}. 

\ssegment{ant1}{}
Although our application is global, this algorithm may equally be constructed in a purely $p$-adic situation.\footnote{
There's a slight caveat: where the algorithm and its subalgorithms take $p$-adic numbers an input, those must be specified by a finite amount of data. This means that the domain of the algorithm must be restricted to those $p$-adic numbers which can be specified by a finite amount of data. However, since we've agreed not to keep track of $\ep$, this restriction on the domain need not concern us any further. 
} In order to minimize the number of decorations, we introduce notation specific to the present situation; these will remain in effect through the proof of Proposition \ref{bee5}.

We let $\Fphi$ denote the category of mixed Tate filtered $\phi$ modules over $\Qp$. We let
\[
\Omega: \Fphi \to \Vect(\Qp)
\]
denote the forgetful functor. As in paragraph \ref{int16}, we let $U(\Zp)$ denote the unipotent part of the Tannakian fundamental group $\Aut^\otimes(\Omega)$. Given $E \in \Fphi$, $v \in \rho(E)$ and $f \in \rho(E)^\lor$, we let 
\[
[E, v, f]
\]
denote the function 
\[
U(\Zp) \to \AA^1_\Qp
\]
given on a point $\ga$ with values in an arbitrary $\Qp$-algebra by
\[
\ga \mapsto f(\ga v).
\]
Recall that $u$ denotes the $\Qp$-point of $U(\Zp)$ associated to the $p$-adic period map. 

Let $D$ be a finite set of elements of $\Zp$ no two of which are congruent modulo $p$. Let $X := \AA^1_{\Zp } \setminus D$. We consider two $\Zp$-integral base points $a,b$ of $X$. We let $_aP_a$ denote the filtered $\phi$ realization of the unipotent fundamental group of $X_\Qp$ at $a$, a unipotent group object of $\Fphi$. We let $_b P_a$ denote the filtered $\phi$ realization of the unipotent path torsor. We let $_a \Uu_a$ denote the completed universal enveloping algebra of $_aP_a$
and let
\[
{_b \Uu_a} := {_a \Uu_a} \times_{_aP_a} {_bP_a}
\]
be the associated rank one free module.  

\Lemma{ant2}
{
We put ourselves in the situation and the notation of segment \ref{ant1}. Consider an element $c$ of the set  $D$ of punctures and a $\Zp$-integral base point $a$ of $X(\Zp)$. Let  $e^c$ denote the element of $_a \Uu_a$ associated to monodromy about $c$. Then we have the equality
\[
u e^c = p 
\left(
\sum_{\eta} 
(\int_c^a \eta)
  \eta
  \right)
\cdot e^c \cdot
\left(
\sum_\om
(\int_a^c \om) \om
\right)
\]
in the noncommutative formal power series ring $_a \Uu_a$. Both sums run over the set of words in the family of differential forms 
\[
\left\{  \frac{dt}{t-d}  \right\}_{d \in  D}
\]
and the integrals are regularized with respect to the unit tangent vector at $c$. The integral of the empty word is defined to be $1$.
}

\bigskip
The proof (which is purely formal) spans segments \ref{ant3}-\ref{ant6}.

\ssegment{ant3}{}
We recall that each path torsor $_bP_a$ possesses a unique $\Qp$-valued point contained in step $0$ of the Hodge filtration \[
p^\dR = {_bp_a^\dR}
\]
and a unique $\Qp$-valued point
\[
p^\cris = {_bp_a^\cris}
\]
fixed by Frobenius. We have
\[
u p^\dR = p^\cris.
\]
We use the de Rham path $_bp^\dR_a$ to identify $_b \Uu_a$ with $_a \Uu_a$. Let us write  $_b \om _a$ for a word $\om$ \textit{regarded as an element of} $_b \Uu_a$. Thus, 
\[
{_b \om_a} =  {_bp_a^\dR} \cdot {_a \om_a}.
\] 
In this notation,
\[
 {_bp_a^\dR} = {_b 1_a}.
\]
We denote $_c e^c_c$ simply by $\ep^c$.

By Besser's definition of the $p$-adic iterated integrals
\[
\int_a^b \om,
\]
we have 
\[
 {_bp_a^\cris} =
  \sum_\om (\int_a^b \om) {_b \om_a}
\]
where the integral of the empty word is defined to be $1$. 

\ssegment{ant4}{}
We have 
\[
\tag{*}
 {_bp_a^\dR} \cdot {_a \om _a} =
  {_b \om _b} \cdot {_bp_a^\dR}.
\]
Indeed, this equality reduces to the case of a one-letter word
\[
\om = {e^c}.
\]
We have
\[
_b e^c_b =
 {_b p^\dR_c} \cdot \ep^c \cdot {_c p^\dR_b}.
\]
Since the composition of de Rham paths is again a de Rham path, it follows that both sides of equation (*) (when $\om = e^c$) are equal to 
\[
{_b p^\dR_c} \cdot \ep^c \cdot {_c p^\dR_a}.
\]

\ssegment{ant5}{}
By segment \ref{ant4}, we have for any word $\om$ in the set $D$ of punctures and any points $a, b, c$:
\[
{_c \eta_b} \cdot {_b \om _a} = 
{_c (\eta \om)_a}.
\]

\ssegment{ant6}{}
We thus have
\begin{align*}
u(_ae_a^c) &=
u( {_ap^\dR_c} \cdot \ep^c \cdot {_c p^\dR_a} )
\\
&= u( {_ap^\dR_c}) \cdot u(\ep^c) \cdot u( {_c p^\dR_a} )
\\
&= 
\left(
\sum_{\eta} 
(\int_c^a \eta)
  {_a\eta_c}
  \right)
\cdot p \ep^c \cdot
\left(
\sum_\om
(\int_a^c \om) 
{_c\om_a}
\right)
\\
&=
{_a
( p
\left(
\sum_{\eta} 
(\int_c^a \eta)
  \eta
  \right)
\cdot e^c \cdot
\left(
\sum_\om
(\int_a^c \om) \om
\right)
)
_a}.
\end{align*}
This completes the proof of Lemma \ref{ant2}.

\ssegment{ant7}{Action of $u$ on an arbitrary loop-word}
Let $\om$ be a word in the set $D$ of punctures regarded as an element of the completed universal enveloping algebra ${_a \Uu_a}$ at a base-point $a$ of $X$. Since the action of $u$ on $_a \Uu_a$  respects multiplication, Lemma \ref{ant2} coupled with the algorithm of \cite{PItInts} for computing $p$-adic iterated integrals provides an algorithm for computing any coefficient of the noncommutative formal power series $u {\om}$ to precision $\ep$.

\ssegment{ant7.5}{Remark on the unipotent nature of the action on loop algebras}
We recall that the generators $e^c$ ($c \in D$) of the completed universal enveloping algebra $_a \Uu_a$ (or ``loop algebra'') have half-weight $-1$ and that the action of $U(\Zp)$ (and hence of $u$) on ${_a\Uu_a}$ is unipotent with respect to the weight filtration. More specifically, the formula of Lemma \ref{ant2} shows that $u \om$ has coefficient $1$ in front of the word $\om$ itself, and that every word which occurs with nonzero coefficient has $\om$ as a subword (by which we mean a subsequence of not necessarily consecutive letters).

\ssegment{ant8}{Action of $u$ on arbitrary path-words}
Let $\om$ be a word in the set $D$ of punctures regarded as an element of the completed universal enveloping bimodule ${_b \Uu_a}$ associated to a pair of base-points $a$ and $b$ of $X$.
The algorithm of segment \ref{ant7} may be upgraded to an algorithm which computes any coefficient of the noncommutative formal power series $u {\om}$ to precision $\ep$, or, which is the same, an algorithm which computes the $p$-adic period
\[
[{_b \Uu_a}, \om, f_{\om'}](u)
\]
 of any matrix entry
\(
[{_b \Uu_a}, \om, f_{\om'}].
\)
 Indeed, 
\(
[{_b \Uu_a}, \om, f_{\om'}](u)
\)
is the $\om'$-coefficient of the noncommutative formal power series 
\[
\tag{*}
u({_b \om_a})
 = u ( {_b p^\dR_a} \cdot {_a \om_a})
 = {_b p^\cris_a} \cdot u({_a \om_a}).
\]
Using the algorithm of segment \ref{ant7}, we expand $u({_a \om_a})$ as a linear combination of words
\[
u({_a \om_a}) = \sum_\eta c_\eta \cdot {_a\eta_a}
\]
with coefficients $c_\eta \in \Qp$ computed to $p$-adic precision $\ep$. We then have
\begin{align*}
[{_b \Uu_a}, \om, f_{\om'}](u) 
&= f_{\om'} (u(_b \om_a) )
\\
&= f_{\om'} ( {_b p^\cris_a} \cdot u({_a \om_a}))
\\
&= f_{\om'} ( {_b p^\cris_a} \cdot
 \sum_\eta c_\eta \cdot {_a\eta_a} )
\\
&= \sum_\eta c_\eta 
f_{\om'} ( {_b p^\cris_a} \cdot  {_a\eta_a} )
\\
&= \sum_\eta c_\eta \int_a^b \om'/\eta
\end{align*}
where the right-division $\om'/\eta$ is defined to be zero whenever $\om'$ is not right-divisible by $\eta$. Using \cite{PItInts} again we compute these last $p$-adic iterated integrals to precision $\ep$. 

\ssegment{bee0.5}{Remark on the unipotent nature of the action on path modules}
We recall that the generators $e^c$ ($c \in D$) of the completed universal enveloping bimodule $_b \Uu_a$ (or ``path module'') have half-weight $-1$ and that the action of $U(\Zp)$ (and hence of $u$) on ${_b\Uu_a}$ is unipotent with respect to the weight filtration. This squares with the computation of segment \ref{ant8}. To see this more clearly, we repeat the computation in slightly different notation: we have
\begin{align*}
u({_b \om _a}) 
	&= {_bp^\cris_a} \cdot u({_a \om_a}) \\
	&= \sum_{\theta} \left( \int_a^b \theta \right) 
	{_b \theta_a} \cdot \sum_\eta c_\eta {_a \eta_a} \\
	&= \sum_{\theta, \eta} \left(
	c_\eta \int_a^b \theta
	\right)
	{_b \theta\eta_a}.
\end{align*}
We find that the coefficient in front of the word $\om = \eta$ ($\theta = 1$) is $1$, and that all words occurring in the sum are left-multiples of words which contain $\om$ as a subword.

\ssegment{bee1}{Elementary remarks on matrices with respect to unordered bases}
Let $k$ be a field, $V$, $W$ finite dimensional vector spaces,
\[
\phi:V \to W
\]
a linear map, $\Vv$ a basis of $V$ and $\Ww$ a basis of $W$. Then the associated matrix is indexed by the set $\Vv \times \Ww$. The entry associated to the pair $(v,w)$ is given by
\[
{_w \phi_v} = w^\lor(\phi v)
\]
where $w^\lor$ is the linear functional on $W$ dual to $w$ with respect to the basis $\Ww$. 

If $V = \bigoplus_i V_i$, $W = \bigoplus_j W_j$ are finite direct sums of finite dimensional vector spaces with bases $\Vv_i$, $\Ww_j$ and $\phi = \bigoplus_{i,j} \phi_{i,j}$ is a direct sum of linear maps
\[
\phi_{i,j}: V_i \to W_j ,
\]
then the matrix associated to $\phi$ is given in terms of the matrices of the $\phi_{i,j}$ as follows: if $v \in \Vv_i$ and $w \in \Ww_j$ then
\[
{_w \phi_v} = {_w (\phi_{i,j})_v}.
\]

\ssegment{bee2}
{Elementary remarks on graded pieces of graded representations}
Let $\gG = \bigoplus_n \gG_n$ be a graded Lie algebra over a field $k$, $E = \bigoplus E_i$ a finite dimensional graded vector space,
\[
\rho: \gG \to \gl E
\]
a graded representation. Let $\Uu$ denote the universal enveloping algebra of $\gG$. Then the induced ring homomorphism 
\[
\rho: \Uu \to \End E
\]
preserves gradings. We spell out what this means. We let 
\[
\End^n E \subset \End E
\]
denote the subspace of homomorphisms which are graded of graded degree $n$:
\[
\End^n E = \bigoplus_i \Hom(E_i, E_{i+n}).
\]
Then $\phi$ sends the $n$th graded piece $\Uu_n$ of $\Uu$ into $\End^n E$. This also means that $\rho$ is compatible with projections, in the sense that the squares
\[
\xymatrix{
\Uu \ar[d] \ar[r] & \End E \ar[d] \\
\Uu_n \ar[r] & \End^n E
}
\]
commute. 

In terms of matrices, the projection has the effect of setting entries in all other graded degrees equal to zero. More precisely, if $\phi \in \End E$ has $n$th graded piece $\phi^n \in \End^n E$, if $\Ee = \bigcup_i \Ee_i$ is a graded basis of $E$ and if $v \in \Ee_i$, $w \in \Ee_j$ are basis vectors of graded degrees $i$ and $j$, respectively, then 
\[
{_w \phi^n_v} = 
\left\{
\begin{matrix}
{_w \phi_v} & \mbox{if } j-i = n \\
0 & \mbox{otherwise.}
\end{matrix}
\right.
\]
To see this, let $P_i$ denote the idempotent  
\[
E \surj E_i \inj E
\]
associated to the $i$th graded piece. Then 
\[
\phi^n(v) = P_{i+n} \phi(v),
\]
so
\begin{align*}
{_w \phi^n_v} &= w^\lor  P_{i+n} \phi(v).
\end{align*}
We complete the verification by noting that 
\[
w^\lor \circ P_{i+n}
 = 
\left\{
\begin{matrix}
w^\lor  & \mbox{if } j = i+n \\
0 & \mbox{otherwise.}
\end{matrix}
\right.
\]

\ssegment{bee2.5}{Universal polynomials for entries of products of graded pieces of the logarithm of a matrix}
Let $\Word(D)$ denote the set of words in the set $D$ of punctures. Let $R$ be the polynomial $\QQ$-algebra 
\[
R = \QQ[\Word D \times \Word D]
\]
graded by setting the degree of a word equal to minus its length as usual. We denote the generator associated to a pair of words $\om, \eta$ by $x_{\om, \eta}$. Let $M$ be the $\Word D \times \Word D$-matrix whose $(\om, \eta)$th entry $_\eta M_\om$ is $1$ if $\om = \eta$, $x_{\om, \eta}$ if $\eta$ contains $\om$ as a subword, and $0$ otherwise. Then the logarithm of $M$ converges (in the sense that each entry is a polynomial). We let
\[
N = \log M.
\]
For any integer $l$, we define a new matrix $^l N$ with entries
\[
{_\eta^l N_\om} = 
\left\{
\begin{matrix}
{_\eta N_\om} & \mbox{if } |\eta|-|\om| = l \\
0 & \mbox{otherwise.}
\end{matrix}
\right.
\]

Let $\om$ be a word of length $n$ in $D$ and let
\[
w = l_1  \cdots  l_r
\]
be a word in the set $\ZZ_{<0}$ of negative integers such that 
\[
l_1 + \cdots + l_s = -n. 
\]
We define a polynomial $P(\om, w) \in R$ by taking the $(\emptyset, \om)$th entry
\[
P(\om, w) = 
{_\om [ {^{l_1}N \dots ^{l_s}N} ]_1}
\]
of the product of graded pieces of $N$ indicated by $\om$.

\ssegment{bee3}{Main algorithm}
We now arrive at the construction of the realization algorithm. As input, the algorithm takes a positive real number $\ep$, a prime number $p$, a finite set $D$ of elements of $\Zp$ and two further elements $a,b$ (given up to $p$-adic precision $\ep$), a natural number $n$, a word $\om$ of length $n$ in the set $D$, and a word $w$ of degree $n$ in the set of symbols
\[
\{v_{-1}, v_{-2}, v_{-3}, \dots \}
\]
indexed and weighted by the negative integers. Distinct points in the set $D \cup \{a,b\}$ must not be congruent modulo $p$, but the elements $a,b$ may belong to the set $D$ (in the latter case, they will be treated as unit tangent vectors). The output consists of a single element
\[
\Aa_{\rm{Real}}(\ep, p, D, a, b, \om, w)
\]
 of $\Qp$ given up to $p$-adic precision $\ep$. To construct it, we first check which variables $x_{\theta, \eta}$ intervene in the polynomial $P(\om, w)$ constructed in segment \ref{bee2.5}. For each such variable, we apply the subalgorithm of segment \ref{ant8} to compute the element 
 \[
 [{_b \Uu_a, \eta, f_{\theta}} ](u)
 \]
of $\Qp$ to precision $\ep$. We then output the value
\[
\Aa_{\rm{Real}}(\ep, p, D, a, b, \om, w) =
P(\om, w) \left(
\left\{
 [{_b \Uu_a, \eta, f_{\theta}} ](u)
\right\}_{\eta, \theta}
\right).
\]

\ssegment{bee4}{}
We now announce the meaning of the output. In terms of the input, we let $X = \AA^1_{\Qp}\setminus D$, and we work with the filtered $\phi$ unipotent path bimodule ${_b \Uu_a}$ on $X$. Let $n$ be the length of the word $\om$. As in segment 4.9 of \cite{mtmue}, we define the \emph{unipotent filtered $\phi$ iterated integral} $I_a^b(\om) \in A(\Zp)_n$ to be the Tannakian matrix entry
\[
I_a^b(\om) = [ {_b \Uu_a}, {_b 1_a}, f_\om]. 
\]
Via the isomorphism 
\[
A(\Zp)_n = \Uu(\Zp)_{-n}^\lor,
\]
the unipotent filtered $\phi$ iterated integral $I_a^b(\om)$ may be \textit{evaluated} at the element $w \in \Uu(\Zp)_{-n}$.

\Proposition{bee5}{
The \emph{realization algorithm} $\Aa_{\rm{Real}}$ halts. Moreover, in the notation of segment \ref{bee4}, its output is within $\ep$ of the $p$-adic number $I_a^b(\om)(w)$. 
}

\begin{proof}
The halting presents no issue. Turning to the verification of the correctness of the output, we fix an arbitrary input datum
\[
(\ep, p, D, a, b, \om, w)
\]
and we set ourselves the task of computing $I_a^b(\om)(w)$ in terms of the periods
\[
[{_b \Uu_a, \eta, f_{\theta}} ](u);
\]
this is mostly a matter of rearranging definitions. Let
\[
{_b \Uu_a^{ \ge -n } } = {_b \Uu_a}/{_b \Uu_a}I^{n+1}
\]
where $I$ denotes the augmentation ideal of ${_a\Uu_a}$. The quotient module ${_b \Uu_a^{ \ge -n } }$ has vector space basis consisting of words of length $\le n$. We note that 
\[
[{_b \Uu_a^{\le -n}, \eta, f_{\theta}} ](u)
=
[{_b \Uu_a, \eta, f_{\theta}} ](u)
\]
so long as $\eta$ and $\theta$ are both of length $\le n$.

Let $_b \rho_a$ denote homomorphism of graded $\Qp$-algebras  
\[
\Uu(\Zp) \to \End {_b \Uu_a^{\le -n}}
\]
induced by the action of the filtered $\phi$ Galois group $U(\Zp)$ on the path bimodule $_b \Uu_a$. Then
\[
I_a^b(\om)(w) = {_\om[{_b\rho_a}(w) ]_1}
\]
is the $(1,\om)$th entry of the matrix associated to the endomorphism ${_b\rho_a}(w)$ of ${_b \Uu_a^{\le -n}}$. To compute it, write $w$ as a product of letters
\[
w = l_1 \cdots l_s
\]
which we identify with the negative integers which parametrize them (while taking care \textit{not} to confuse the above juxtaposition of letters with the product of integers). Meanwhile, recall that for $i$ a negative integer, 
\[
v_i = {^i(\log u)}
\]
is the $i$th graded piece of $\log u$. Thus, if we set $M$ equal to the matrix associated to ${_b \rho_a}(u)$, and $N = \log M$, we have
\[
{_b \rho_a}(v_i) =  {^i N}.
\]
Consequently,
\begin{align*}
{_b \rho_a}(w) 
	&= {_b \rho_a}(l_1) \cdots {_b \rho_a}(l_s) \\
	&= {^{l_1} N} \cdots {^{l_s} N}.
\end{align*}
Putting the pieces back together, we have
\begin{align*}
I_a^b(\om)(w) 
	&= {_\om[{_b\rho_a}(w) ]_1} \\
	&= P(\om, w)
		\left(
		\left\{
		{_\theta M_\eta}
		\right\}_{\eta, \theta} 
		\right) \\
	&\sim_\ep P(\om, w)
		\left(
		\left\{
		\widetilde
		{_\theta M_\eta}
		\right\}_{\eta, \theta} 
		\right) \\
	&= \Aa_{\rm{Real}}(\ep, p, D, a, b, \om, w)
\end{align*}
where
$\widetilde
		{_\theta M_\eta}$
denotes the $\ep$-approximation produced by the algorithm of segment \ref{ant8} and $\sim_\ep$ signals an error bounded by $\ep$ (up to an admissible change in $\ep$). 
\end{proof}

\segment{6}{Basis algorithm}

\ssegment{6.1}{}
We now construct an algorithm which takes as input an open integer scheme
\[
Z \subset \Spec \Oo_K,
\] 
a prime $p$ of $\ZZ$, a natural number $n$, and an
\[
\ep \in p^\ZZ,
\]
and returns the following data.
\begin{enumerate}
\item
An open subscheme $\Zo \subset Z$. We write
\[
\overline S = \{\qQ_1, \dots, \qQ_{\bar s} \}
\]
for its complement.
\item
Sets
\[
\Ee^g_1 = \{\log^U \al_{1,1}, \dots, \log^U \al_{1,r_1+r_2-1}\},
\]
\[
\Ee^r_1 = \{\log^U \be_1, \dots, \log^U \be_{\bar s} \}
\]
of unipotent logarithms of elements of $\Oo_\Zo^*$.
\item
 For each integer $n' \in [2, n]$,
\begin{enumerate}
\item
a set of single-valued unipotent polylogarithms
\[
\widetilde \Ee_\npr = 
\{\Li^{U, sv}_\npr(a_{n', 1}), \dots, \Li^{U, sv}_\npr(a_{n', e_{n'}}) \},
\]
where $e_m$ denotes the dimension of the motivic extension space \[
\Ext^1_\Zo \big( \QQ(0), \QQ(m) \big),
\]
\item
a set $\Pp_\npr$ of unipotent iterated integrals of half-weight $n'$,
\item
an $\ep' \in p^\ZZ$,
\item
an algorithm which takes a pair $I,J$ of unipotent iterated integrals of half-weight $n'$ as input and returns a rational number
\[
\langle I,J \rangle_{\ep'} \in \QQ.
\] 
\end{enumerate}
\end{enumerate}
We denote this algorithm by $A_\m{Basis}$. We first announce the meaning of its output in proposition \ref{ABasis}; we then construct the algorithm in segments \ref{13f}--\ref{T5}, and prove the proposition in segments \ref{T7}--\ref{T8}.

\Proposition{ABasis}{
\begin{enumerate}
\item
Suppose $A_\m{Basis}(Z, p, n, \ep)$ halts. Then we have:
\begin{enumerate}
\item
$\Ee^g_1$ forms a basis of $A(\Spec \Oo_K)_1$.
\item
$\Ee^g_1 \cup \Ee^r_1$ forms a basis of $A(\Zo)_1$.
\item
Each $\widetilde \Bb_\npr:= \widetilde \Ee_{n'} \cup \Pp_\npr $ ($n' =2, 3, \dots, n$) forms a basis for a subspace $\widetilde B_\npr$ of $A(\Zo)_\npr$ complementary to the space $D_\npr$ of decomposables. Moreover, the space $P_\npr$ spanned by $\Pp_\npr$ is disjoint from the space $E_\npr$ of extensions.
\item
Relative to this basis, the projection $\Ee_\npr$ of $\widetilde \Ee_\npr$ onto $E_\npr$ forms a basis of $E_\npr$.
\item
We let $\Bb_\npr = \Ee_\npr \cup \Pp_\npr$, we let $\Dd_\npr$ denote the set of monomials in $\Bb_{<n'}$, we let
\[
\Aa_{n'} = \Bb_\npr \cup \Dd_\npr,
\]
and we denote by $|\cdot|_{\Aa}$ the norm induced on $A(\Zo)_\npr$ by the basis $\Aa_\npr$.
If $\Li^{E, sv}_\npr(a_{n', i})$ denotes the projection of $\Li^{U, sv}_\npr(a_{n', i})$ onto $E_\npr$ then we have
\[
\left| \Li^{U, sv}_\npr(a_{n', i}) - \Li^{E, sv}_\npr(a_{n', i}) \right|_{\Aa}
< \ep'.
\]
\item
We have
$
\ep' \le \ep.
$
\item
If $I,J$ are unipotent iterated integrals of half-weight $n'$, and 
\[
\langle I,J \rangle_{\Aa}
\]
denotes the inner product in which the basis $\Aa_\npr$ is orthonormal,
 then
\[
\left| \langle I,J \rangle _{\ep'} - \langle I,J \rangle_{\Aa} \right|_p < \ep'.
\]
In other words, the algorithm produced as part (d) of the output of $A_\m{Basis}$ computes this inner product up to precision $\ep'$.
\end{enumerate}
\item
If Zagier's conjecture (conjecture \ref{int14}), Goncharov exhaustion (conjecture \ref{int15}) and the Hasse principle for finite cohomology (condition \ref{hassest}) hold for $n$, $Z$, and $K$, then the computation $A_\m{Basis}(Z, p, n, \ep)$ halts.
\end{enumerate}
}

\ssegment{13f}{}
We write $d_G$ for the reduced Goncharov coproduct, regarded as a map
\[
\CUI(Z)_r \to (\CUI(Z)_{>0})^{\otimes 2}_r.
\]

\ssegment{02_a}{}
The algorithm $A_\m{Basis}$ searches arbitrarily through the countably-infinite set of data $(\Zo, \widetilde \Ee_{\le n}, \Pp_{\le n}, \ep')$. For the rest of the construction, we fix such a datum, and construct an algorithm which returns a boolean argument, as well as a function $\langle \cdot, \cdot \rangle _{\ep'}$. If the boolean result if {\it False}, we start over with a new datum. If the boolean result is {\it True}, we output
\[
(\Zo, \widetilde \Ee_{\le n}, \Pp_{\le n}, \ep', \langle \cdot, \cdot, \rangle_{\ep'}).
\]

 For the base case with $n'=1$ we require our basis to be of the form given in the proposition, with
\[
\{ a_{1,1}, \dots, a_{1,r_1+r_2-1} \}
\]
a basis for $\Oo_K^*$, and each $b_i$ a generator for a power of $\qQ_i$.

\ssegment{T1}{}
We assume for a recursive construction that conditions (a)--(f) have been verified in half-weights $<n'$, and that the algorithm computing the inner products 
$\langle I,J \rangle_{\ep'}$ has been constructed in half-weights $<n'$.   The inner products give us maps
\[
\widetilde {^UI_{\Aatil}}: \CUI(\Zo)_r \to A(\Zo)_r
\]
and
\[
\widetilde {^UI_{\Aatil}}^{\otimes 2}:
 \CUI(\Zo)^{\otimes 2}_r \to A(\Zo)^{\otimes 2}_r
\]
in the form of explicit matrices with rational coefficients with respect to the bases $\Aa_{<n'}$ of iterated integrals already constructed in lower half-weights.

\ssegment{06_a}{}
We check if the divisors associated to the iterated integrals in
$
\Bb_{n'}
$
are \'etale over $\Zo$; if not, we return {\it false}. 

\ssegment{T2}{}
We check each element $I$ of $\widetilde \Ee_\npr$ for proximity to $E_\npr$. To do so, we lift $I$ to an integrand $w \in \CUI(\Zo)_\npr$, compute $\widetilde{^UI_{\Aatil}}(d_G(w))$, and check that the $p$-adic norm of the resulting vector is $< \ep'$. If not, we return {\it False}.

\ssegment{T3}{}
We check
\[
\IStdd(\widetilde \Ee_\npr)
\]
for $\ep'$-linear independence in the sense of segments \ref{2}, \ref{3} using the \emph{realization algorithm} of segment \ref{4}. If this fails, we return {\it False}.

\ssegment{T4}{}
We check $d(\Pp_\npr \cup \widetilde \Dd_\npr)$ for $\ep'$-linear independence by lifting $\Pp_\npr$, $\widetilde \Dd_\npr$ to $\CUI(\Zo)$ and applying $\IAtil^{\otimes 2} \circ d_G$. If this fails, we return {\it False}. 

\ssegment{T4.5}{}
We check that $\ep'$ is sufficiently small compared to the spread of the basis $\Aatil_\npr$ that the projection onto the space $E_\npr$ of extensions will preserve the linear independence of $\widetilde \Aa_\npr$. If this fails, we return {\it False}. Otherwise we return {\it True}.

\ssegment{T5}{}
For the inner product in half-weight $n'$, it suffices to construct
\[
\langle w,x \rangle_{\ep'}
\]
for $x \in \CUI(\Zo)_\npr$ arbitrary and $w \in \widetilde \Aa_\npr$ a basis element. We first construct the inner products
\[
\set{ \left\langle w, x \right\rangle }
{w \in \Pp_\npr \cup \Ddtil_\npr}.
\]
It may happen that 
$\widetilde{^UI_{\Aatil}}^{\otimes 2} \big( d_G(w) \big)$
 is not in the span $V$ of the set
\[
\Vv :=
\IAtil^{\otimes 2}\big(d_G (\Pp_\npr \cup \Ddtil_\npr)\big).
\tag{$*$}
\]
Nevertheless, we may compute the projection $w'$ of $w$ onto $V$ with respect to the basis $(\Aatil^{\otimes_2})_\npr$. Subsequently, expanding $w'$ in the set $\Vv$ is a matter of solving a system of linear equations with rational coefficients;
 the linear independence of ($*$) established in step \ref{T4} above, implies the uniqueness of the solution. 

To compute the remaining inner products 
\[
\set{ \left\langle w, x \right\rangle_{\ep'} }
{w \in \Eetil_\npr},
\]
we replace $x$ by
\[
x' = x - \sum_{w\in \Pp_\npr \cup \Ddtil_\npr} 
 \left\langle x, w \right\rangle_{\ep'} w.
\] 
We then compute the projection of $\IStdd(x')$ onto the span of $\IStdd(\Eetil_\npr)$ inside $\prod_{\pP|p} A(\Op)$. The linear independence of the latter, established in segment \ref{T3} above, ensures that the resulting system of linear equations will have a unique solution.

 This completes the construction of the algorithm.

\ssegment{T7}{}
We now prove proposition \ref{ABasis}. Suppose as in part (1) of the proposition that $A_\m{Basis}(Z,p,n, \ep)$ halts. Parts (a), (b) are clear. For parts (c) and (d) we note that if $\ep$-approximations are $\ep$-linearly independent, then the actual vectors are linearly independent. Part (e) is clear, except for perhaps the admissibility of the change in $\ep'$; see segment \ref{met_lin} below. For part (f) we of course limit ourselves to searching through data satisfying $\ep' \le \ep$ in the first place. 

For part (g), we note that the square below, left, commutes.
\[
\xymatrix{
A(\Zo)_r \ar[r]^-d &
(A(\Zo)^{\otimes 2})_r &
A(\Zo)_r \ar[r]^-d  &
(A(\Zo)^{\otimes 2})_r  \\
\CUI(\Zo)_r \ar[u]^-{{^UI}}  \ar[r]_-{d_G} &
(\CUI(\Zo)^{\otimes 2})_r  \ar[u]_-{{^UI}^{\otimes 2}}&
\CUI(\Zo)_r \ar[u]^-{\IAtil} \ar[r]_-{d_G} &
(\CUI(\Zo)^{\otimes 2})_r \ar[u]_-{\IAtil^{\otimes 2}}
}
\]
Since the corresponding vertical arrows in the left and right squares differ by $\ep'$, it follows that the square on the right fails to commute by at most $\ep'$. This gives us the inequality of proposition \ref{ABasis}(g) up to a possible change in $\ep'$ stemming from the failure of $\IStdd$ to respect two splittings: the splitting of
\[
\widetilde E_r \subset A(\Zo)_r
\]
given by the complementary space $P_r \oplus D_r$ inside the source on the one hand, and the splitting of 
\[
\IStdd(\widetilde E_r) \subset \prod_{\pP |p} A(\Op)
\]
induced by the standard basis inside the target on the other hand.\footnote{In fact, to decrease the change in $\ep$, we could replace the standard basis of $A(\Op)$ with a basis compatible with the decomposition of the latter into extensions, primitive non-extensions, and decomposables, as we do for $A(\Zo)$. The map $\IAtil$ would then be nearly compatible with the splittings, yielding a function which is quadratic in $\ep$.} This is clearly admissible; we omit the details.

\ssegment{met_lin}{}
Returning to part (e), we must show that our modifications of $\ep$ form an algorithmically computable function which goes to zero with $\ep$. This is elementary, and fits into the general setting of a valued field $(k, |\cdot|)$ and linear map
\[
\phi: k^m \to k^n
\]
with kernel $E$. We claim that if
\[
|\phi x| < \ep
\]
then
\[
| x - E| < C\ep
\]
for some algorithmically computable constant $C$. We let $W$ denote the image of $\phi$ and $V$ the coimage, both with induced norms. For $x \in k^m$ we let $\bar x$ denote its image in $V$, and we let $\bar \phi$ denote the isomorphism
\[
V \xto{\sim} W
\]
induced by $\phi$. We fix a metric isomorphism $V = W$ arbitrarily (in practice this would be accomplished by constructing orthonormal bases of both spaces), and we let $C\inv$ be the absolute value of the smallest eigenvalue. Then for $x \in k^m$ we have
\begin{align*}
|\phi x| 
&= | \bar \phi \bar x| \\
& \ge C\inv |\bar x | \\
& = C\inv | x -E|,
\end{align*}
independently of the choice of metric isomorphism, which establishes the claim. 

\ssegment{T8}{}
We turn to part (2) of the proposition: the conditional halting. If the conjectures of Zagier and Goncharov hold for the given input, then our search-space includes an open subscheme $\Zo \subset Z$, a basis $\Ee_{\le n}$ of $E_{\le n}$ consisting of single valued unipotent ($\le n$)-logarithms which are combinatorially-unramified over $\Zo$, and a linearly independent set $\Pp_{\le n}$ of unipotent iterated integrals which are combinatorially-unramified over $\Zo$ completing $\Ee_{\le n} \cup \Dd_{\le n}$ to a basis of $A(\Zo)_{\le n}$. Our claim is that if the Hasse principle holds, then for $\ep'$ sufficiently small, the boolean subalgorithm evaluated on the associated datum
\[
(\Zo, \Ee_{\le n}, \Pp_{\le n}, \ep')
\]
returns {\it True}. The map
\[
 \Re_p: 
\Qp \otimes \Ext^1_K(\QQ(0), \QQ(n))
\to
\prod_{\pP | p} \Ext^1_{\Op}(\Qp(0), \Qp(n))
\]
from the global motivic Ext group to the product of filtered $\phi$ Ext groups corresponds to the localization map of the Hasse principle through the $p$-adic regulator isomorphism of Soul\'e on the source (\cite{SouleHigher}) and through the Bloch-Kato exponential map (\cite{BlochKato}) on the target. 
So under the Hasse principle, the map
\[
 \Real_p:
A(\Zo)_n \to 
\prod_{\pP | p} A(\Op)_n
\]
(for $n\ge 2$) is injective near the extension space $E_n$. For any set of linearly independent vectors in a (finite dimensional) normed vector space, there exists an $\ep$ such that any set of $\ep$-approximations is $\ep$-linearly independent. So the claim follows. This concludes the proof Proposition \ref{ABasis}. 

\begin{rmk}
Recall that the iterated integrals through which we search to form the set $\Pp_{\le n}$ are parametrized by families of sections of $\PP^1$ over $\Zo$ (and have no particular relationship to $\thrpl$). Moreover, the Goncharov exhaustion conjecture places no bound on the height of points needed to form a basis. Thus, the search, as formulated here, is huge and unwieldy. Nevertheless, as a second attempt to convince the reader of the termination, we note that no matter how we order this countably huge set through which we search, if, as predicted by the conjectures, a successful candidate exists, it will, in due course, be met. In practice, we would probably place a height bound $B$ on points and a bound $C$ on the size of $Z \setminus \Zo$ and then, in turns, increase $B$, decrease $\ep'$, increase $C$. Beyond that, as mentioned in the introduction, pairing down the data and ordering it for an efficient search presents an interesting problem in its own right.
\end{rmk}

\segment{8a}{Change of basis algorithm}

\ssegment{8a.1}{}
We now construct an algorithm which changes the basis constructed by the basis algorithm to one which is compatible with the coproduct up to possible errors of size $\ep$. This algorithm takes as input a datum $(\Zo, \widetilde \Ee_{\le n}, \Pp_{\le n}, \ep, \langle \cdot, \cdot \rangle_\ep)$ as in the output of the basis algorithm $A_\m{Basis}$, and outputs for each $n' \le n$, a square matrix of size $a_{n'} \times a_\npr$ over $\QQ$. We denote this algorithm by $A_\m{Change}$ and the resulting $n'$th matrix by
\[
A_\m{Change}(\Zo, \widetilde \Ee_{\le n}, \Pp_{\le n}, \ep, \langle \cdot, \cdot \rangle_\ep, n').
\]

\ssegment{8b}{}
For each $n' \le n$ we fix a set
\[
\Si^o_\npr = \{\si_{-n', 1}, \dots, \si_{-n', e_\npr} \}
\]
of symbols, and define the half-weight of $\si_{-n', i}$ to be $-n'$. We may then speak about words in $\Si^o_{\ge -n}$ and about the half-weight of a word. Entries in our matrix will be indexed by pairs $(I, w)$, with
\[
I \in \widetilde \Aa_\npr = 
\widetilde \Ee_\npr  \cup \Pp_\npr \cup \widetilde \Dd_\npr
\]
($\widetilde \Dd_n$ denoting the set of monomials in $\widetilde \Bb_{<n} = \widetilde \Ee_{<n} \cup \Pp_{<n}$ of half-weight $n$ as usual), and $w$ a word in $\Si^o$ of half weight $-n'$. We construct the associated matrix entry
$
a_{I,w}
$
by recursion on the length of $w$. We write $\widetilde \Ee_\npr$ as a vector
\[
\Eetil_\npr = (\Eetil_{n', 1}, \dots, \Eetil_{n', e_\npr})
\]
(so $\Eetil_{n',i}= \Li_\npr^{U,sv}(a_{n',i})$ is a single-valued unipotent $n'$-logarithm). When $w = \si_{-n',i}$ is a one-letter word, we set
\[
a_{I,w} = 
\left\{
\begin{matrix}
1 & \mbox{if } I = \Eetil_{n',i} \\
0 & \mbox{otherwise}.
\end{matrix}
\right.
\]
Now suppose $n' = l+m$ with $l,m >0$, and let $w$ be a word of half-weight $-m$. Using the Goncharov coproduct and the inner product, we expand 
\[
d_{l,m} I = \sum c_{j,k} \, \Aatil_{l,j} \otimes \Aatil_{m,k}
\]
in the basis
\[
\Aatil_{l} \otimes \Aatil_{m} = 
\left\{ 
\Aatil_{l,j} \otimes \Aatil_{m,k}
\right\}_{j,k}
\]
of $A_l \otimes A_m$ to precision $\ep$. In terms of the $c_{j,k}$, we define
\[
a_{I, \si_{-l,i} \cdot w}
=
\sum_{j,k} c_{j,k} \cdot 
a_{\Aatil_{l,j}, \si_{-l,i}} \cdot 
a_{\Aatil_{m,k}, w}
\]
which equals
\[
\sum_k c_{i,k} a_{\Aatil_{m,k},w}
\]
if we number the basis $\Aatil_{m}$ in such a way that 
\[
\Aatil_{m,j} = \Eetil_{m,j}
\]
for $j \in [1,e_m]$.

\ssegment{8c}{}
We now make the meaning of the output precise. Let $\Ee_{n,i}$ denote the projection of $\Eetil_{n,i}$ onto the space $E_n$ of extensions (so in the context of the basis algorithm, we have $\Ee_{n,i} = \Li_n^{E,sv}(a_{n,i})$). For each $n$, we let $\Bb_n = \Ee_n \cup \Pp_n$, and we let $\Dd_n$ denote the set of monomials in $\Bb_{<n}$. According to proposition \ref{0},
\[
\Aa_n := \Bb_n \cup \Dd_n
\] 
forms a basis of $A(\Zo)_n$. Let $\si_{-n,i} \in \Uu(\Zo)_{-n}$ denote the element dual to $\Ee_{n,i}$ relative to this basis. With this interpretation of the set $\Si^o$ of symbols $\si_{-n,i}$, according to proposition \ref{.5}, $\Si^o$ becomes a set of free generators of the free prounipotent group $U(\Zo)$. For $w \in \Uu(\Zo)_{-n}$ a word in $\Si^o$ of half-weight $-n$, let $f_w \in A(\Zo)_n$ denote the corresponding function.

Let $(b_{w,I})_{w,I}$ denote the inverse of the matrix constructed in the algorithm. For $w$ a word of half-weight $-n$, define $\ftil_w \in A(\Zo)_n$ by
\[
\ftil_w = \sum_{I \in \Aatil_n} b_{w,I} I.
\]

\Proposition{8d}{
In the situation and the notation above, we have 
\[
| f_w - \ftil_w | < \ep.
\]
}

\begin{proof}
This is equivalent (up to an admissible change in $\ep$) to the estimate
\[
| a_{I,w} - \langle w,I \rangle) | < \ep.
\]
Our algorithm is based on the following two properties of the numbers $\langle w, I \rangle$ for $w$ a word in our set of abstract generators $\Si$, and $I$ an element of our concrete basis $\Aatil_n$ :
\begin{enumerate}
\item we have
\[
\left|
\langle \si_{n,i} ,I \rangle -
\left\{
\begin{matrix}
1 & \mbox{if } I = \Eetil_{n,i} \\
0 & \mbox{otherwise}
\end{matrix}
\right\}
\right|
< \ep,
\]
and
\item
we have
\[
\left|
\langle \si_{n,i}\cdot w, I \rangle
-
\langle \si_{n,i} \otimes w, dI \rangle
\right|
< \ep.
\]
\end{enumerate} 
The proposition follows.
\end{proof}

\section{Construction of geometric algorithms}

\segment{9a}{Cocycle-evaluation-image algorithm}
Fix finite sets $\Si_{-1}$, $\Si_{-2}$, $\Si_{-3}, \dots$, $\Si_{-n}$ and $\Si^\circ_{-1}$, $\Si^\circ_{-2}$, $\Si^\circ_{-3}, \dots$, $\Si^\circ_{-n}$ with 
\[
\Si_{-1} \subset \Si^\circ_{-1}
\]
and $\Si^\circ_{i} = \Si_{i}$ for $i \le -2$.
Set
\[
\Si = \bigcup \Si_i,
\qquad
\Si^\circ = \bigcup \Si^\circ_i.
\]
Let $\nN(\Si)$, $\nN(\Si^\circ)$ denote the free graded pronilpotent Lie algebras on generators $\Si$, $\Si^\circ$. As usual, we refer to the grading as the \emph{half-weight}. Let $\nN^\PL$ denote the polylogarithmic Lie algebra over $\QQ$,
\[
\nN^\PL = \QQ(1) \ltimes \prod_{i=1}^\infty \QQ(i)
\]
with $\QQ(i)$ in half-weight $-i$. We write $\Hom_{\LLie}^\Gm$ for homogeneous Lie-algebra homomorphisms of graded degree $0$.
Let $\phi$ denote the natural quotient map
\[
\phi: \nN(\Si^\circ) \surj \nN(\Si)
\]
and let $\ev$ denote the map
\[
\Hom_{\LLie}^\Gm(\nN(\Si), \nN^\PL) \times \nN(\Si^\circ)_{\ge -n}
\to
\nN^\PL_{\ge -n} \times \nN(\Si^\circ)_{\ge -n}
\]
given by
\[
\ev(\Cc, F) = (\Cc(\phi(F)), F).
\]
Then $\ev$ is in an obvious sense a map of finite dimensional affine spaces, and it is straightforward to construct an algorithm which computes its scheme-theoretic image. (At the very least, this would be a standard application of elimination theory, but in fact, it should be possible to obtain a closed formula.) We omit the details. We refer to this as the \emph{cocycle-evaluation-image} algorithm. We denote it by $\Aa_\m{Eval}$, and its output, a finite list of elements of
\[
S^\bullet ( \nN^\PL_{\ge -n} \times \nN(\Si^\circ)_{\ge -n})^\lor,
\]
by $\Aa_\m{Eval}(\Si, \Si^\circ,n)$.

\segment{9b}{Chabauty-Kim-loci algorithm}

\ssegment{9b'}{}
We now construct the Chabauty-Kim-loci algorithm discussed in the introduction. As input it takes an open integer scheme $Z$, a prime $p$ of $\ZZ$, a natural number $n$, and an $\ep \in p^\ZZ$. As output it returns a finite family\[
\widetilde \Bb = \widetilde \Ee \cup \Pp
\]
of unipotent iterated integrals, and a finite family $\{\widetilde F_i\}_i$ of elements of the polynomial ring
\[
\QQ[\widetilde \Bb, \log, \Li_1, \Li_2, \dots, \Li_n],
\]
which we denote by $\Aa_\m{Loci}(Z, p, n, \ep)$.

\ssegment{9c}{}
We run $A_\m{Basis}(Z,p,n,\ep)$. This gives us our set $\widetilde \Bb = \widetilde \Bb_{\le n}$ of unipotent iterated integrals. We run $\Aa_\m{Change}$ on $A_\m{Basis}(Z,p,n,\ep)$ to obtain a matrix 
\[
M_{\le n} = \bigoplus_{i=0}^n M_i.
\]

\ssegment{9d}{}
We let $\Si_{-1}$ denote a set of size $e_1=\dim \Oo_Z^* \otimes \QQ$, $\Si_{-1}^o$ a set containing $\Si_{-1}$ of size $e_1^o=\dim \Oo_{\Zo}^*\otimes \QQ$, and for each $i \in [2,n]$, $\Si_{-i} = \Si^o_{-i}$ a set of size
\[
e_i = \dim \Ext^1_K( \QQ(0), \QQ(i)).
\]
We run $\Aa_\m{Eval}(\Si, \Si^o, n)$ to obtain a finite family $\{F^\m{abs}_i\}_i$ of elements of $S^\bullet(\nN^\PL_{\ge -n} \times \nN(\Si^o)_{\ge-n})^\lor$.

\ssegment{9e}{}
We pull back along the quotient map
\[
\nN(\Si^o) \surj \nN(\Si^o)_{\ge -n}.
\]
We pull back further along the logarithm
\[
U(\Si^o) \to \nN(\Si^o)
\]
Denoting the natural coordinates on $\nN^\PL$ by $\log, \Li_1, \Li_2, \Li_3, \dots$, we obtain a finite family of elements of 
\[
S^\bullet\nN^\PL_{\ge -n} \otimes A(\Si^o) = 
A(\Si^o)[\log, \Li_1, \dots, \Li_n]
\]
which are contained in degrees $\le n$.

\ssegment{9f}{}
The matrix $M_{\le n}$ defines a linear bijection
\[
A(\Si^o)_{\le n} \xto{\sim} \QQ[\widetilde \Bb]_{\le n}
\]
which we use to obtain the hoped-for family $\{ \widetilde F_i\}_i$. This completes the construction of the algorithm.

\ssegment{10a}{Proof of Theorem \ref{MainTh}}
We have a sequence of maps
\begin{align*}
U(X)^\PL_{\ge -n, \Qp}
 \to
&  \nN(X)^\PL_{\ge -n} \times \Spec \Qp
\\
 & \to \nN(X)^\PL_{\ge -n} \times \nN(\Zo)
\surj \nN(X)^\PL_{\ge -n} \times \nN(\Zo)_{\ge -n}
\end{align*}
and an associated sequence of Cartesian squares:
\[
\xymatrix{
\Hom^\Gm_{\Lie}(\nN(Z), \nN(X)^\PL_{\ge -n})
 \times \nN(\Zo)_{\ge -n}
 \ar[r]^-{\overline \ev_n}
&
\nN(X)^\PL_{\ge -n} \times \nN(\Zo)_{\ge -n}
\\
\Hom^\Gm_{\Lie}(\nN(Z), \nN(X)^\PL_{\ge -n} )
 \times \nN(\Zo)
\ar[u] \ar[r]^-{\ev_n}
&
\nN(X)^\PL_{\ge -n} \times \nN(\Zo)
\ar[u] 
\\
\Hom^\Gm_{\Lie}(\nN(Z), \nN(X)^\PL_{\ge -n}) \times \Spec \Qp
\ar[r]^-{\ev_n(\nu_{\pP})} \ar[u]
&
\nN(X)^\PL_{\ge -n} \times \Spec \Qp
\ar[u]
\\
\Hom^\Gm_{\m{Groups}} (U(Z), U(X)^\PL_{\ge -n})_\Qp
\ar[r]^-{ev_n(u_{\pP})} \ar[u]
&
U(X)^\PL_{\ge -n, \Qp}.
\ar[u] \ar@/_7ex/@<-5ex>[uuu]_\Gamma
}
\]
We write $\Im$ for scheme-theoretic image.
 By the mere commutativity of the outer square we have a containment of closed subschemes 
\[
\Im ev_n(u_\pP) \subset \Gamma\inv \Im \overline \ev_n.
\]
In view of Propositions \ref{ABasis} and \ref{8d}, this establishes part 1 of the theorem (correctness of the output upon halting).

For the conditional halting, we contend that formation of scheme-theoretic image along each horizontal map is compatible with pullback along each vertical map. We start at the top. For $\nu \in \nN(\Zo)$ mapping to $\overline \nu \in \nN(Z)$ and
\[
\phi: \nN(Z) \to \nN(X)^\PL_{\ge -n}
\]
a $\Gm$-equivariant homomorphism, $\phi(\overline \nu)$ depends only on the image of $\nu$ in $\nN(\Zo)_{\ge -n}$. So
\[
\Im \ev_n = 
 \left( \Im \overline \ev_n \right) \times \nN(\Zo)_{<-n}.
\]
We go on to the middle square. By the period conjecture, $A(\Zo) \to \Qp$ is flat. (In general, if $A$ is integral, $K$ a field, and $\psi:A \to K$ is injective, then $\psi$ is flat, since it factors as a localization map followed by a field extension.) Hence formation of the scheme-image is compatible with pullback. Turning to the bottom square, the vertical maps are iso, so this is clear. This completes the proof of Theorem \ref{MainTh}.

\section{Construction of analytic algorithm}
\label{Newt}

\segment{17b}{}
We now construct an algorithm for deciding whether the number of zeroes of a $p$-adic power series in a given ball is zero or one, given a sufficiently close approximation. This is quite standard, but we are unaware of a reference for this exact problem.

More precisely, we construct a boolean-valued algorithm which takes as input a boolean $b \in \{0,1\}$, natural numbers $N$, $r$ and $h$, and a polynomial with rational coefficients
\[
\widetilde F = \sum_{i=0}^N \atil_i T^i
\]
which has at most $b$ ($\Qp$-rational) roots inside the disk of radius $p^{-r}$. We call this algorithm the \emph{root criterion algorithm} and denote the output by $\Aa_\m{RC}(b,N,r,h,\Ftil)$. We first announce the meaning of the output as a remark. 

\disp{Remark}{17c}{
In our application below,
\[
F = \sum_{i=1}^\infty a_iT^i
\]
will be a power-series expansion of a polynomial in logarithms and polylogarithms of half-weight $h$ over $\Qp$, and $\Ftil$ will be an approximation of $F$ with arithmetic precision $p^{-r}$ and geometric precision $e^{-N}$. Suppose $\Aa_\m{RC}(b,N,r,h,\Ftil) =$ {\it True} . Then $F$ has at most $b$ roots within the disk of radius $p^{-r}$. This amounts to an elementary use of Newton polygons, together with the growth estimate
\[
v(a_k) \ge -h\log_p(k) 
\]
which follows from proposition 6.7 of Besser--de Jeu \cite{Lip}.
}

\segment{17d}{Case 1: b=0}
\begin{enumerate}
\item
If $\atil_0 =0$, return {\it False}.
\item
Compute real solutions to
\[
v(\atil_0) -rt = -h\log_p t
\]
to within $0.5$. If there are none, return {\it True}. 
\item
Otherwise, there are two solutions $t_L < t_R$. If $t_R >N$, return {\it False}.
\item
Check the condition
\[
v(\atil_k) > v(\atil_0) -rk
\]
for $1 \le k \le t_R$. If the condition holds, return {\it True}. Otherwise return {\it False}.
\end{enumerate}

\segment{17e}{Case 2.1: $b=1$, $\atil_0 = 0$ }
\begin{enumerate}
\item
If $\atil_1=0$, return {\it False}.
\item
Compute solutions to 
\[
-r(t-1)+v(\atil_1) = -h\log_p t
\]
to within $0.5$. If there are none, return {\it True}.
\item
Otherwise, there are two solutions $t_L < t_R$. If $t_R > N$, return {\it False}.
\item
Check the condition
\[
v(\atil_k) > -r(k-1) +v(\atil_1)
\]
for $2 \le k \le t_R $. If this condition holds, return {\it True}. Otherwise, return {\it False}.
\end{enumerate}

\segment{17f}{Case 2.2: $b=1$, $\atil_0 \neq 0$}
In this case, we have 
\[
\tag{$*$}
v(\atil_1) = v(\atil_0) -r 
\]
and
\[
\tag{$**$}
v(\atil_0) -rt = -h \log_p t
\]
has two solutions $t_L < t_R$.
\begin{enumerate}
\item
If $t_R > N$, return {\it False}. 
\item
Check the condition
\[
v(\atil_i) > v(\atil_0) -ri
\]
for $2 \le i \le t_R$. If this condition holds, return {\it True}, otherwise return {\it False}.
\end{enumerate}

\section{Numerical approximation}
\label{Num}


We present the results of the unpublished work \cite{PItInts} concerning computation of $p$-adic iterated integrals on the projective line. For background we refer the reader to Furusho \cite{FurushoI, FurushoII} and to Chatzistamatiou \cite{ChatInt}. The results of this section should be compared with prior and with concurrent works by Besser -- de Jeu \cite{Lip}, by Jarossay \cite{JarssayExplicitI}, and by \"Unver \cite{UnverCyclotomic}.

\segment{171}{}
Let $K$ denote a finite extension of $\Qp$. A general p-adic iterated integral on the line with arbitrary punctures with good reduction is a multiple polylogarithm up to sign: if the points $a_1, \dots, a_m$ of $\Oo_K^*$ lie in distinct residue disks, then we have
\begin{align*}
\int_{1_0}^{a_{m+1}}
 \left(
 \frac{dt}{t}
 \right)
 ^{n_m-1}
 &
\frac{dt}{t-a_m} \cdots
\left(\frac{dt}{t}\right)^{n_1-1}
\frac{dt}{t-a_1} 
\\
&=
(-1)^m \Li_{n_1, \dots, n_m} 
\left( \frac{a_2}{a_1}, \frac{a_3}{a_2}, \dots, \frac{a_{m+1}}{a_m} \right)
\end{align*}
(see theorem 2.2 of Goncharov \cite{GonMPMTM}).
Here $1_0$ denotes the tangent vector $1$ at $0$. We will restrict attention to the case $m=1$, $a_1 = 1$, i.e. to iterated integrals on $\thrpl$; in terms of multiple polylogarithms, this means restricting attention to multiple polylogarithms of one variable. We do this merely for concreteness: the difficulty in going from this case to the general case is largely a notational one.

\segment{172}{}
We work with the rigid analytic space equal to the closed unit disk minus the residue disk about $1$:
\[
U = \Spm K\{t,u\} / \big( u(t-1) -1 \big),
\]
with log structure induced by the divisor $\{0\}$.\footnote{For us, this merely means that we will be working with differential forms with log poles at the origin; we will not make use of log geometry.} The coordinate ring $\Oo(U)$ has a map 
\[
\rho_0: \Oo(U) \to B:=K\{t\}
\]
as well as a map
\[
\rho_1: \Oo(U) \to A:= K\{w,u\}/(wu-1)
\]
sending
\[
t \mapsto w+1
\]
(re-centering the unit circle). Both $\rho_0$ and $\rho_1$ are injective, and on a practical level, all computations that occur within the algorithm may in fact be carried out inside the rings $A$ and $B$ up to given arithmetic and geometric precision. There will be no need to verify convergence (or overconvergence) algorithmically. As usual, we do not keep track of the error incurred. 

\segment{180}{Remark}
There are several possible treatments of the singular points $0, 1, \infty$. In one approach, we puncture over the residue field $k$ (which corresponds, via Berthelot's theory of rigid analytic tubes, to removing residue disks over $\Spm K$). In another approach, we include the points $0,1, \infty$, allowing log poles instead of puncturing; in Berthelot's construction of isocrystals, this does not require the use of log geometry. This approach is convenient for dealing with tangential base-points. (There is also the possibility of working with the rigid analytic thrice punctured line $\PP^1_{\Spm K}\setminus\{0,1,\infty\}$, needed when considering points of bad reduction.) The resulting categories of unipotent isocrystals on the special fiber or of unipotent connections on any of the above rigid analytic spaces are all canonically equivalent. Our use of the rigid analytic space $U$ means that we choose to mix the first two approaches. 

\segment{181}{}
We give a brief sketch of the definition of the p-adic multiple polylogarithms we wish to evaluate. We refer to Besser \cite{BesserHeidelberg, Besser} and to Chatzistamatiou \cite{ChatInt} for more thorough accounts.

We let $\unVIC(U)$ denote the category of unipotent vector bundles with integrable connection with log poles at $0$. Its equivalence with, say, the category of unipotent isocrystals on $\PP^1_k\log 0\setminus\{1,\infty\}$ endows it with a Frobenius pullback functor 
\[
F^*: \unVIC(U) \to \unVIC(U).
\]
If $\om$, $\om'$ are fiber functors \emph{compatible with Frobenius}, then $F^*$ acts on the Tannakian path torsor
\[
_{\om'}P_\om = \Isom^\otimes(\om, \om').
\]
The exact notion of compatibility is a bit subtle, and we refer the reader to Chatzistamatiou \cite{ChatInt} for a detailed discussion. The result of this detailed discussion however, is that in our setting, a naive approach will be sufficient.

We consider the $K$-rational fiber functors
\[
\unVIC(U) \xyto{\om_{1_0}}{\om_x} \Vect K
\]
associated to the tangent vector $1_0$ and to $x$ respectively. According to Besser \cite{BesserHeidelberg}, there's a unique Frobenius fixed point $p^\cris \in ({_xP_{1_0}})(K)$. 

We let $\Etil$ denote the KZ-connection over $U$: $\Etil = \Oo_U \langle \langle e_0, e_1 \rangle \rangle$, with connection given by
\[
\nabla (W) = e_0W \frac{dt}{t} + e_1 W \frac{dt}{1-t}
\]
for any word $W$ in $e_0, e_1$. Thus, if
\[
\Ll = \sum_W \Ll_W W
\]
is an arbitrary section, we have
\[
\nabla(W) = d\Ll + e_0 \Ll \om_0 + e_1 \Ll \om_1
\]
where 
\begin{align*}
\om_0 := \frac{dt}{t}, &&
\om_1 &:= \frac{dt}{t-1},
\end{align*}
and
\[
d\Ll = \sum_W \Ll'_W Wdt
\]
is given by differentiating the coefficient functions.

 Each $K$-rational fiber of $\Etil$ is canonically equal to
\[
\Uu: = K \langle \langle e_0, e_1 \rangle \rangle;
\]
let $\la_W$ denote the linear functional 
\[
\Uu \to K
\]
associated to the word $W$. The Tannakian path $p^\cris$ gives rise to a linear map
\[
p^\cris_{\Etil}: \Etil(1_0) \to \Etil(x).
\]
In terms of the path $p^\cris$ and the linear functional $\la_W$, we define the p-adic multiple polylogarithms we wish to evaluate by 
\[
\Li_W(x) = \la_W \big( p^\cris_{\Etil}(1) \big).
\]

\segment{173}{}
Our algorithm depends on a notion of \emph{residue over residue disk}. In turn, this depends on an elementary lemma:

\begin{lemma*}
Suppose we have a congruence relation 
\begin{align*}
\sum_{i,j \ge 0} a_{i,j} w^iu^j \equiv 
\sum_{i,j \ge 0} b_{i,j} w^iu^j 
\mod (wu - 1)  
\end{align*}
in the restricted formal power series ring $K\{w,u\}$. Then we have an equality of convergent sums
\[
\sum_{j=0}^\infty a_{j+1,j} = 
\sum_{j=0}^\infty b_{j+1,j}
\]
in the nonarchimedean field $K$.
\end{lemma*}

\begin{proof}
We provide the details of this elementary verification. In a nonarchimedean field, a sum whose terms tend to zero converges, so the convergence is immediate. For the equality, suppose the congruence relation is witnessed by the equation
\[
\sum_{i,j \ge 0} (a_{i,j} - b_{i,j}) = 
(wu-1)\sum_{i,j \ge 0} c_{i,j}w^iu^j.
\] 
We then have for each $n, k \in \NN$, 
\[
c_{n+k+1, k+1} = 
\sum_{r=1}^k a_{n+r, r} - \sum_{r = 1}^k b_{n+r,r}.
\]
Restricting attention to $n=1$ and taking the limit as $k \to \infty$, we obtain the result. 
\end{proof}

Returning to our definition of the residue, we consider the space
\[
\Om^1(U) = \Oo(U) \frac{dt}{t}
\]
of rigid analytic 1-forms with log poles at the origin. The pullback of an arbitrary element 
\[
\om = f(t,u) \frac{dt}{t}
\] 
along the map $\rho_1$ defined in segment \ref{172} is given by
\[
\rho^*_1\om =  (w+1)\inv f(w+1, w\inv) dw.
\]
We define
\[
\Res_{D(1)} \om := \Res_{w < 1}(w+1)\inv f(w+1, w\inv),
\]
where
\[
\Res_{w<1} : K\{w,u\}/(wu-1) \to K
\]
is defined by
\[
\Res_{w<1} \left(
\sum_{i,j} a_{i,j}w^iu^j 
\right) := \sum_{j=0}^\infty a_{j+1,j}.
\]

\segment{174}{}
We will construct elements $\tau_1 \in \Uu$ and $L \in \Etil(U)$ recursively. To do so, we set 
\begin{align*}
\om_1' &:= \frac{pt^{p-1} dt}{t^p -1}.
\end{align*}
We denote the length of a word $W$ by $|W|$. We also let $|W|_i$ denote the number of occurrences of the letter $e_i$. We set 
\[
\tag{$\sharp$}
\tau_1 := pe_1 +
 \underset{|W|_1 \ge 1}{\sum_{|W| \ge 2}}
  \tau_W W
\]
and
\[
\tag{$\natural$}
L = 1 + \sum_{|W| \ge 1} L_W W.
\]
We require that $L$, which will be determined by a system of differential equations, satisfy the initial condition 
\[
\tag{$\flat$}
L_W(0) = 0
\]
for all nonempty words $W$. We construct the functions $L_W$ as elements of the ring $A$. There, we set $t := w+1$. We may transport the 1-forms $\om_0$, $\om_1$, $\om_1'$ to $A$ in the obvious way. We may also formally differentiate as well as take residues about the open disk $|w|<1$, which we continue to denote by $\Res_{D(1)}$. A 1-form $\om$ in the free rank-one module $Adw$ has a primitive in $A$ if and only if $\Res_{D(1)} \om = 0$. With these notations, and this last fact in mind, we define $\tau_W$ in terms of lower terms by 
\[
\tag{*}
\tau_W = 
-\Res_{D(1)} \left(
p(L_{W/e_0} - L_{e_0 \backslash W}) \om_0
+
\underset{ W', W'' \neq \emptyset }{\sum_{W = W'W''}}
L_{W'} \tau_{W''} \om_1 - L_{e_1 \backslash W} \om_1'
\right).
\]
Notationally, the term with the left-division $e_1 \backslash W$ is equal to $0$ if $W$ is not left-divisible by $e_1$. We define $L_W$ in terms of lower terms by the differential equation
\[
\tag{$\star$}
dL_W = p(L_{W/e_0} - L_{e_0 \backslash W}) \om_0
+
\underset{  W'' \neq \emptyset }{\sum_{W = W'W''}}
L_{W'} \tau_{W''} \om_1 - L_{e_1 \backslash W} \om_1',
\]
the equation being solvable over $A$ since equation (*) above guarantees that the right hand side has no residue. Again, terms with an impossible left or right division are to be interpreted as zero.

\segment{174.5}{}
We now use the elements $\tau_W \in K$ to construct more elements $\tau^{V}_W \in K$ indexed by pairs of words $V$, $W$ with 
\[
V \subset W
\]
by which we mean that $V$ occurs as an ordered subsequence of (not necessarily contiguous) letters. To do so, we let $\tau$ be the endomorphism of $\Uu$ determined by 
\[
\tag{$\sharp$}
\tau(e_0) = pe_0
\]
and
\[
\tag{$\natural$}
\tau(e_1) = \tau_1.
\]
We then define $\tau^{V}_W$ by
\[
\tag{$\flat$}
\tau(V) = \sum_{W \supset V} \tau^{V}_W W.
\]

\segment{175}{}
We construct certain rational numbers $c_{i,j,W}$ ($i,j \in \NN$, $W$ a word in $e_0, e_1$) which arise from the KZ-connection. The $k$th power of the covariant derivative applied to a word $W'$:
\[
\left( \frac{\nabla^k_{\partial/\partial t} W'}{k!} \right)
\]
is of the form
\[
\underset{|W|_1 \le j}
{\underset{|W|_0 \le i}
{\sum_{i+j=k}}} 
\frac{c_{i,j,W}}{t^i(t-1)^j} WW'.
\]
To compute the coefficients $c_{i,j,W}$ algorithmically, we simply apply
\begin{align*}
\frac{\nabla_{\partial/\partial t}}{i+j+1}
&
\left( \frac{1}{t^i(t-1)^j}W' \right)
= 
\frac{1}{i+j+1} \left(
\frac{-i}{t^{i+1}(t-1)^j} I 
\right.
\\
&+ 
\left.
\frac{-j}{t^{i}(t-1)^{j+1}} I 
 + \frac{1}{t^{i+1}(t-1)^j}e_0 +\frac{0}{t^{i}(t-1)^{j+1}}e_1 
\right) W'
\end{align*}
iteratively and collect terms. 

\segment{176}{}
All ingredients above are independent of the endpoint. We now fix the point $x \in \thrpl(\Oo_K)$ at which we wish to evaluate. In terms of the rational coefficients constructed in segment \ref{175}, and in terms of a lift $\si$ of Frobenius to $K$, we define for each word $W$, an element $\ep_W(x) \in K$ by the infinite sum
\[
\ep_W(x)= 
 \sum_{\set{i,j \in \NN}{ i \ge |W|_0, \; j \ge |W|_1} } 
\frac
{c_{i,j,W}(x^p-x)^{i+j}}
{(x^\si)^i(x^\si-1)^j}
. 
\]
We will discuss its convergence below.

\Theorem{178}{
(Joint with Andre Chatzistamatiou.)
Let $T$ be a word in $\{e_0, e_1\}$, let $p$ be a prime, let $K$ be a finite extension of $\Qp$ and let $x \in \thrpl(\Oo_K)$. Then the p-adic multiple polylogarithm $\Li_T^p(x)$ is given in terms of the values $L_W(x)$, in terms of the constants $\tau^V_W$, in terms of the values $\ep_W(x)$, and in terms of multiple polylogarithms of lower weight, by
\[
\Li_T^p(x) 
=
(-1)^{|T|_0}
(1-p^{|T|})\inv 
\underset{V \neq T}
{
\underset{W \supset V}
{\sum_{T=UW'W}}
}
(-1)^{|V|_0}
\epsilon_U(x) L_{W'}(x) \tau^V_W \Li^p_V(x).
\]
(In particular, the terms $\ep_U(x)$ and $L_{W'}(x)$ appearing on the right converge.)
}

\segment{179}{Proof of Theorem \ref{178}}

\ssegment{182}{}
Let $\phi$ denote the Frobenius lift 
\[
U \to U
\]
over $\si$ given by $\phi(t) = t^p$ and let $\phi'$ be a lift of Frobenius which fixes the endpoint $x$. For instance, when $K = \Qp$ we may set $\phi'(t) = (t-x)^p+x$. (Over)convergence gives rise to a canonical isomorphism 
\[
\ep: \phi^* \Etil \xto{\sim} \phi'^* \Etil.
\]
We recall that the pullback along the semilinear Frobenius lift $\phi$ is accomplished in two steps as a base-change along $\si$ followed by a $K$-linear pullback, as in the familiar diagram
\[
\xymatrix{
U \ar[r]^{\phi/K} \ar[dr] &
U^\si \ar[d] \ar[r]^\si  &
U \ar[d] \\
&
\Spm K \ar[r]_\si &
\Spm K
}
\]
and similarly for $\phi'$. In concrete terms, Frobenius-invariance of $p^\cris$ means that the square
\[
\xymatrix{
(\phi^*\Etil)(\oneato)
 \ar[r]^{p^\cris_{\phi^*\Etil}}
  \ar@{=}[d]
   &
(\phi^*\Etil)(x) \ar[d]^{\ep(x)} \ar[d]_\sim 
\\
\si^*\Etil(\oneato) \ar[r]_{\si^*(p^\cris_{\Etil})}
	 &
\si^*\Etil(x)
}
\]
commutes.

\ssegment{183}{}
Together with the choice of unit vector $1 \in \Etil(1_0)$, the KZ-connection corepresents the fiber functor: for any unipotent connection $E$, we have 
\[
\Hom(\Etil, E) = E(1_0).
\]
This is proved in a surprising way (using complex iterated integrals) by Kim in \cite{kimii}, and in a more straightforward way (via a certain iterative construction of universal extensions which has appeared in various places in the literature) for instance by Chatzistamatiou \cite{ChatInt}. This gives us a canonical identification between the fiber $\widetilde E(\oneato)$ and the completed universal enveloping algebra of the unipotent fundamental group. It also means that there's a unique overconvergent horizontal morphism 
\[
\theta: \Etil \to \phi^*\Etil
\]
such that
\[
\theta(1_0): \Etil(1_0) \to (\phi^*\Etil)(1_0)
 = \si^* \big( \Etil(1_0) \big)
\]
sends $1 \mapsto \si^*1$. 

\ssegment{184}{}
Let $\Li(x)$ denote the noncommutative formal power series 
\[
\Li(x) = \sum_W \Li_W(x)W.
\]
Placed side by side, $\theta$, $\ep$, and $p^\cris$ form a commutative diagram like so.
\[
\xymatrix{
\Etil(\oneato)
	\ar[r]^{p^\cris_{\Etil}} \ar[d]_{\theta(\oneato)} 
	\POS p+(-8,6)*+{1}="ul" 
&
\Etil(x) 
	\ar[d]^{\theta(x)}
	\POS p+(8, 6)*+{\Li(x)}="ur"  
\\
(\phi^*\Etil)(\oneato) \ar[r]^{p^\cris_{\phi^*\Etil}}
 \ar@{=}[d]  
&
(\phi^*\Etil)(x) \ar[d]^{\ep(x)} \ar[d]_\sim 
\\
\si^*\Etil(\oneato) \ar[r]_{\si^*(p^\cris_{\Etil})}
	\POS p+(-8,-6)*+{\si^*1}="ll" 
&
\si^*\Etil(x) \POS p+(8,-6)*+{\si^*\Li(x)}="lr"
\ar@{|->}@/^5pt/"ul";"ur"
\ar@{|->}@/_10pt/"ul";"ll"
\ar@{|->}@/_5pt/"ll";"lr"
}
\]
We will see that $\Li(x)$ is uniquely determined by the equation
\[
\si^*\Li(x) = \ep(x) \theta(x) (\Li(x)).
\]

\ssegment{191}{}
It follows from the analysis of tangential fiber functors in Chatzistamatiou \cite[\S3.5]{ChatInt} that there's a canonical isomorphism from the fiber functor $\om_{1_0}$ to the functor
\[
\om_0: E \mapsto E(0)
\]
which is compatible with the action of our chosen Frobenius lift $\phi$. We claim that the fiber $\theta(0)$ of $\theta$ at $0$, regarded as a map 
\[
\Uu \to \si^*\Uu
\]
is equal to $\tau$, and that the value $\theta(U)(1)$ of $\theta$ at the identity element $1$ of $\Etil(U)$ is equal to $L$. These facts follow from two key properties of $\theta$, which in turn follow from a certain functorial characterization of $\theta$. If
\[
f: T \to U
\]
is a rigid analytic space over $U$, and $E$ is a quasi-coherent sheaf over $U$, we set 
\[
E(T):= \Gamma(T, f^*E). 
\]
We let $\Vect T$ denote the category of vector sheaves. We let $\om_T$ denote the functor
\[
\Vect T \from \unVIC(U \log 0)
\]
induced by $f$ (forget the connection and pull back). We let $\om_{0_T}$ denote the composite
\[
\Vect T \from \Vect K \xfrom{\om_0} \unVIC(U). 
\]
We let $\phi_*\om_T$ denote the composite
\[
\Vect T \xfrom{\om_T} 
\unVIC(U) \xfrom{\phi^*}
\unVIC(U),
\]
and similarly for $\om_{0_T}$. Then we have canonical isomorphisms 
\[
\Etil(T) = \Hom(\om_{0_T}, \om_T),
\]
and
\[
(\phi^*\Etil)(T) = \Hom(\phi_*\om_{0_T}, \phi_* \om_T),
\]
which are natural in $T$. Moreover, translated through these isomorphisms, $\theta$ sends a 2-morphism
\[
\om_{0_T} \to \om_T
\]
to its composite with the 1-morphism $\phi^*$. It follows, on the one hand, that the fiber $\theta(0)$ is the $\si$-linear ring homomorphism induced by the action of Frobenius on the unipotent fundamental group, and on the other hand, that  the map of global sections $\theta(U)$ is equivariant for the right-action of $\Etil(0)$ on $\Etil(U)$ and for the right action of $\phi^*\Etil(0)$ on $\phi^*\Etil(U)$ through $\theta(0)$. This last property means that for any $g \in \Uu$, we have 
\[
\tag{Equivariance}
\theta(U)(g) = \theta(U)(I) \cdot \theta(0)(g).
\]
Thus, $\theta$ is completely determined by two small pieces: a \textit{horizontal} piece $\Ll:= \theta(U)(1)$, and a \textit{vertical} piece $\Tt: = \theta(0)$. Moreover, $\Tt$ is determined by its action on the generators $e_0$ and $e_1$. Finally,
\[
\Tt(e_0) = pe_0
\]
and $\Tt_1:= \Tt(e_1)$ has constant term $0$ and linear term $pe_1$.

\ssegment{192}{}
We will now show that $L = \Ll$ and $\tau = \Tt$. We first obtain the initial condition for $\Ll$:
\[
\Ll(0) = \theta(U)(1)(0) = \theta(0)(1(0)) = 1. 
\]
Here `$1$' denotes the empty word, regarded first as a section of $\Etil$ over $U$, and then as a section of $\phi^*\Etil$ over $U$. Its \emph{value} $1(0)$ is the unit element of the fiber $\Etil(0)$, which gets sent to the unit element of $(\phi^*\Etil)(0)$ since $\phi(0)$ is a homomorphism. 

By the horizontality of $\theta$, we have the equation 
\[
\tag{$\sharp$}
\nabla' (\theta(U)(1)) = \theta(U)(\nabla 1)
\]
in the $\Oo(U)$-module 
\[
\Gamma
\big(U, (\phi^* \widetilde E) \otimes \Om^1_U \big)
=
(\phi^* \widetilde E)(U) \otimes \Om^1(U)
= 
(\phi^* \widetilde E)(U) \frac{dt}{t}.
\]
Here, $\nabla'$ denotes the pullback of $\nabla$ along the Frobenius lift $\phi$. By the equivariance property, this equation may be rewritten in terms of $\Ll$ and $\Tt$ as follows. We denote the free generators $\phi^*e_i$ of $\phi^* \widetilde E$ simply by $e_i$. We let $\om_i'$ denote the pullback of $\om_i$ by the Frobenius lift $\phi$. With this notation, we have:
\begin{align*}
\tag{$\natural$}
d\Ll + \sum_{i=0,1} e_i \Ll \om_i'
	&= \nabla \Ll \\
	&= \theta(U)(\nabla(1)) \\
	&= \theta(U) \left( \sum_{i=0,1} e_i \om_i \right) \\
	&= \sum_{i=0,1} \theta(U)(e_i) \om_i \\
	&\overset{(\mbox{equivariance})}{=}
	 \sum_{i=0,1} \Ll  \cdot \Tt(e_i(0)) \om_i.
\end{align*}
Plugging in $pe_0$ for $\Tt(e_0)$ and 
\[
\Tt_1 = pe_1 + \sum_{|W| \ge 2} \Tt_W W
\]
for $\Tt(e_1)$, we obtain 
\[
\tag{$\flat$}
d \Ll + p e_0 \Ll \om_0 + e_1 \Ll \om_1'
=
p \Ll e_0\om_0 + \Ll \Tt_1 \om_1.
\]
Modulo the augmentation ideal $I$, equation ($\flat$) becomes 
\[
d \Ll_\emptyset = 0.
\]
Hence, from the initial condition, we find that
\[
\Ll_\emptyset =1.
\]
Thus $\Ll$ has the form
\[
\Ll = 1 + \sum_{|W| \ge 1} \Ll_W W.
\]
Projecting equation ($\flat$) onto the $W$-coordinate for an arbitrary word $W$, we find that the functions $L_W := \Ll_W$ satisfy equation \ref{174}($\star$). Applying $\Res_{D(1)}$ to both sides of equation \ref{174}$\star$, we find that the constants $\tau_W := \Tt_W$ satisfy equation \ref{174}(*). This completes the verification that $\Ll = L$ and $\Tt = \tau$.

An overconvergent 1-form $\om$ on $U$ with log poles at the origin has an overconvergent primitive if and only if $\Res_0 \om = 0$ and $\Res_{D(1)}\om = 0$. It follows that $L$ is overconvergent. In particular, the values $L_W(x)$ converge.

\ssegment{193}{}
To compute $\ep$, we write
\begin{align*}
\epsilon(\phi^*W')
&=
 \sum_{k=0}^{\infty} 
 (\phi^\sharp t -\phi'^\sharp t)^k 
 \phi'^* 
 \left( \frac{\nabla^k_{\partial/\partial t} W'}{k!} \right)
\\
	&= \sum_k (\phi^\sharp t -\phi'^\sharp t)^k
	\cdot
	{\phi'}^*
	\Big(	
	\underset{|W|_1 \le j}
	{\underset{|W|_0 \le i}
	{\sum_{i+j=k}}}
	\frac {c_{i,j, W}}
		{ t^i(t-1)^j}
		WW'
		\Big) 
\\
	&= \underset{|W|_1 \le j}
	{
	\underset{|W|_0 \le i}
	{\sum_{i,j \in \NN}}
	} 
		\frac{
		c_{i,j,W} 
		\big(\phi^\sharp t -\phi'^\sharp t \big)^k
		}
		{{\phi'}^\sharp 
		\big( t^i(t-1)^j \big)
		}
		\phi'^*(WW').
\end{align*}
The coefficient of ${\phi'}^*(WW')$ is independent of $W'$; setting $\ep_W$ equal to this coefficient and valuating at $t=x$ we obtain the formula given in segment \ref{176}. The prounipotence of $\Etil$ implies convergence of $\ep$. In particular, the values $\ep_W(x)$ converge. 

\ssegment{194}{}
Setting
\[
v= \sum_T v_T T,
\]
the equation
\[
\epsilon(x) \theta(x)(v) = v
\]
becomes
\begin{align*}
\sum_T v_T T
	&= \epsilon(x) \theta(x) \left( \sum_V v_V V \right) 
\\
	&= \sum_V v_V \epsilon(x) 
	\left( \underset{W \supset V}{\sum_{W'}}
	 L_{W'}(x) \tau^V_W W'W \right) 
\\
	&= \sum_V v_V \underset{W \supset V}
	{\sum_{W'}} L_{W'}(x) \tau^V_W 
	\sum_U \epsilon_U(x) UW'W
\\
	&= \underset{W \supset V}{\sum_{V,U,W'}} 
	\epsilon_U(x) L_{W'}(x) \tau^V_W v_V UW'W
\\
	&= \sum_T 
	\left( 
	\underset{W\supset V}
	{\sum_{T = UW'W}}
	 \epsilon_U(x) L_{W'}(x) \tau^V_W v_V
	  \right) T. \\
\end{align*}
Hence
\begin{align*}
v_T 
	&=  \underset{W\supset V}{\sum_{T = UW'W}} \epsilon_U(x) L_{W'}(x) \tau^V_W v_V  
\\
	&= p^{|T|} v_T + \underset{V \neq X}
	{\underset{W\supset V}{\sum_{T = UW'W}}} \epsilon_U(x) L_{W'}(x) \tau^V_W v_V .
\end{align*}
So the equations
\[
v_I = 1
\]
\[
\epsilon(x)\theta(x)(v) = v
\]
are equivalent to
\[
v_I = 1
\]
and
\[
(1-p^{|T|})v_T = 
\underset{V \neq X}{\underset{W \supset V}{\sum_{T=UW'W}}}
\epsilon_U(x) L_{W'}(x) \tau^V_W v_V
\]
from which we obtain the formula for  
\[
\Li_T(x) = v_T
\]
given in the theorem.\footnote{Actually, differing sign conventions necessitate the addition of the signs seen in the final formula to accord with the traditional definition.} This completes the proof of theorem \ref{178}.

\segment{15a}{}
We apply this to the computation of $p$-adic multiple zeta values. We have the path composition formula
\[
\int_x^z \om_n \cdots \om_1 = 
\left( 
\sum_{i=0}^n \int_y^z \om_n \cdots \om_{i+1} 
\cdot \int_x^y \om_i \cdots \om_1
\right)
\]
and the path reversal formula
\[
\int_y^x \om_n \cdots \om_1 = (-1)^n \int_x^y \om_n \cdots \om_1. 
\]
If $f: Y\to X$ is an isomorphism of rigid curves, then
\[
\int_{f(x)}^{f(y)}\om_n \cdots \om_1 = \int_x^y (f^*\om_n) \cdots (f^*\om_1).
\]
Applying this to $\om_i \in \left\{ \frac{dx}{x}, \frac{dx}{1-x} \right\}$, to
\[
x \mapsto 1-x
\]
on $\PP^1$, and to an auxiliary point $y$, we obtain
\begin{align*}
\int_y^{\vec{-1_1}} \om_n \cdots \om_1
	&= (-1)^n  \int^y_{\vec{-1_1}} \om_n \cdots \om_1 \\
	&= \int_{\vec{1_0}}^{1-y} \om_n^\circ \cdots \om_1^\circ
\end{align*}
where $\left( \frac{dx}{x}\right)^\circ = \left(  \frac{dx}{1-x} \right)$ and $\left(  \frac{dx}{1-x} \right)^\circ = \left(  \frac{dx}{x} \right)$. So
\begin{align*}
\ze(W) 
	&= \int_{\vec{1_0}}^{\vec{-1_1}} \om_W \\
	&= \sum_{W = W''W'} \int_y^{\vec{-1_1}} \om_{W''} 
		\int_{\vec{1_0}}^y \om_{W'} \\
	&=  \sum_{W = W''W'} \Li_{W''^\circ}(1-y)(\Li_{W'}y).
\end{align*}

\section{The equation-solving algorithm}
%

\segment{17g}{Construction of the algorithm}

\ssegment{17h}{}
We now construct the promised algorithm for totally real fields. Our algorithm takes as input an open integer scheme $Z$ and outputs a finite set of elements of $X(Z)$. We denote the output by $\Aa_\m{ES}(Z)$. As explained in the introduction, the success of the algorithm depends on first finding $X(Z)$ by a naive search, and then proving that there are no other points by verifying that $X(Z_\pP)_n = X(Z)$. Recall also (from our formulation of the convergence conjecture (\ref{int8})) that success (i.e. halting) depends also on the absence of repeated roots for $n$ sufficiently large.

\ssegment{17i}{}
We find a prime $p$ of $\ZZ$ in the image of $Z$, for which $Z$ is totally split. We fix arbitrarily a prime $\pP$ of $Z$ lying above $p$.

\ssegment{17j}{}
Our algorithm searches through the set of triples $(n, N, \ep)$, $n, N \in \NN$, $\ep$ in a countable subset of $\RR_{> 0}$ with accumulation point $0$.  After each attempt, we increase $n$ and $N$ and decrease $\ep$. To each such triple, our algorithm assigns a set $X(Z)_n$ of points of $X(Z)$ and a boolean. If the boolean output is {\it True}, then we output $X(Z)_n$. If the boolean output is {\it False}, then we continue the search. 

To produce the set $X(Z)_n$, we spend $n$ seconds searching for points.\footnote{
Evidently, this choice is arbitrary. In reality we would probably search up to a chosen height bound $B(n)$ which grows to $\infty$ as $n$ goes to $\infty$.
} To produce the boolean output, we follow the steps described in segments \ref{17k}--\ref{17o} below.

\ssegment{17k}{}
Partition $X(\Op)$ into $\ep$-balls, decreasing $\ep$ as needed to ensure that each ball contains at most one element of the set $X(Z)_n$ (our, potentially incomplete, list of integral points). 

\ssegment{17l}{}
Run $\Aa_\m{Loci}(Z,p, n, \ep)$ to obtain a family $\{\Ftil_i\}_i$ of polylogarithmic functions. Symmetrize the family with respect to the $S_3$-action using the formulas \ref{int7}(*). Set $h_i$ equal to the half-weight of $\Ftil_i$.

\ssegment{17m}{}
We focus our attention on an $\ep$-ball $B$ containing a rational representative $y \in B$. Expand each polylogarithmic function $\Ftil_i$ to arithmetic precision $\ep$ and geometric precision $e^{-N}$ about $y$; denote the result by $\widetilde {\widetilde F^\pP_i}$.

\ssegment{18a}{}
Fixing $i$, write
\[
\widetilde {\widetilde F^\pP_i} = \sum_{j=0}^N \widetilde{\atil_j}T_j.
\]
Check the following condition:

\bigskip
\begin{center}
{\it 
For each $i$ and each $j \le N$, if $\widetilde{\atil}_j \neq 0$ then
$\left| \widetilde{\atil}_j \right| \ge \ep$.
}
\end{center}

\bigskip
\noindent
If this {\it fails}, return, {\it False}.

\ssegment{17n}{}
We continue to work with the single $\ep$-ball $B$.
Set $b$ equal to the number of points ($0$ or $1$) in $X(Z)_n \cap B$. Choose an $r \in \NN$ such that $\ep \ge p^{-r}$. Run the root-criterion algorithm $\Aa_\m{RC}(b, N, r, h_i, {\widetilde{\widetilde {F^\pP_i}}})$ for varying $i$. 

\ssegment{17o}{}
Repeat the steps of segments \ref{17m}--\ref{17o} for each $\ep$-ball $B$. If for each ball $B$ there exists an $i$ such that 
\[
\Aa_\m{RC}(b, N, r, h_i, \widetilde{\Ftil_i^\pP}) = \m{\it True},
\]
return {\it True}. Otherwise return {\it False}. This completes the construction of the algorithm.

\segment{17r}{Equation-solving theorem}
We come to the main applications announced in the introduction.

\sdisp{Theorem}{17s}
{
Let $Z$ be an open integer scheme with fraction field $K$.
\begin{enumerate}
\item
Suppose the algorithm $\Aa_\m{ES}(Z)$ halts. Then we have
\[
\Aa_\m{ES}(Z) = X(Z).
\]
\item
Assume $K$ is totally real. Suppose Kim's conjecture (conjecture \ref{int8}) holds for $Z$ at level $n$. Suppose Zagier's conjecture (\ref{int14}) holds for $K$ and $n' \le n$. Suppose Goncharov's conjecture (\ref{int15}) holds for $Z$ and $n' \le n$. Suppose the period conjecture holds for the open subscheme $\Zo \subset Z$ constructed in segment \ref{6} in half-weights $n' \le n$. Suppose $K$ obeys the Hasse principle for finite cohomology (condition \ref{int19}) in half-weights $2 \le n' \le n$. Then $\Aa_\m{ES}(Z)$ halts.
\item
Assume $K = \QQ$. Suppose Kim's conjecture holds for $Z$ at level $n$. Suppose Goncharov's conjecture holds for $Z$ and $n' \le n$. 
Suppose the period conjecture holds for the open subscheme $\Zo \subset Z$ constructed in segment \ref{6} in half-weights $n' \le n$.
  Then $\Aa_\m{ES}(Z)$ halts.  
\end{enumerate}
}

\begin{proof}
Parts (1) and (2) are a direct application of theorem \ref{MainTh}. One point may require clarification: the role of segment \ref{18a}. Let $\{F_i^\pP\}_i$ denote the generators of the Chabauty-Kim ideal close to $\{\Ftil_i^\pP\}_i$ whose existence is guaranteed by theorem \ref{MainTh}. Fixing an $\ep$-ball $B$ with representative $y \in B$ and an $i$, write
\[
F_i^\pP = \sum_{j=1}^\infty a_jT^j
\qquad
\mbox{and}
\qquad
\Ftil^\pP_i = \sum_{j=0}^\infty \atil_jT^j
\]
for the power series expansions about $y$. Then after an admissible change in $\ep$ (depending on $N$), we have
\[
\left| a_j - \atil_j \right| < \ep
\]
for all $j \le N$. By construction, we have
\[
\left| \atil_j - \widetilde{\atil_j} \right| < \ep, 
\]
hence
\[
\left| a_j - \widetilde{\atil_j} \right| < \ep.
\]

For part (1) of the theorem, suppose that for given $B$, $i$ and $j$, we find that $|\widetilde{\atil_j}| \ge \ep$. Then by the nonarchemedian triangle inequality, we have 
\[
|\widetilde{\atil_j}| = |a_j|. 
\]
This means that those valuations whose precise determination is needed for the root criterion algorithm $\Aa_\m{RC}$, will indeed be precise. For part (2), we need only note that for $\ep$ sufficiently small depending on $N$, we will indeed have for each $\ep$-ball $B$, each $i$, and each $j \le N$, either $\widetilde{\atil_j} = 0$ or $|\widetilde{\atil_j}| \ge \ep$.

We turn to part (3). The period conjecture implies in particular that the $p$-adic zeta values $\ze^p(n')$ for $n' \in [3,n]$ odd are nonzero. In turn, this implies that the unipotent zeta values
\[
\ze^U(n') \in \Ext^1\big( \QQ(0), \QQ(n') \big)
\]
are nonzero. Since these extension groups have dimension 1 while the remaining extension groups have dimension 0, this would imply  Zagier's conjecture in this case. (In fact, since the extension groups for $n \ge 2$ don't depend on $Z$, we're free to choose any prime $p$. Choosing a regular one, where the nonvanishing is known, we obtain one possible proof of Zagier's conjecture for this case.) 

We claim that the period conjecture also implies the Hasse principle. Indeed, the Hasse principle in this case merely says that a linear map from a vector space of dimension $\le 1$ is injective. For this, it's enough to show that the map is nonzero. However, the composite map
\[
\Qp \otimes \Ext^1_\ZZ \big( \QQ(0), \QQ(n') \big) 
\to
\Ext^1_{\Zp} \big( \Qp(0), \Qp(n') \big)
\to
\Qp
\]
sends $\ze^U(n')$ to $\ze^p(n')$ which, as we've already noted, is nonzero if the period conjecture holds. 
\end{proof}

\section{Beyond totally real fields}
\label{Beyond}

\segment{18b-}{}
We first explain the inadequacy of the methods developed above for dealing with fields which are not totally real. 

\begin{prop*} \label{+}
Let $Z$ be an open integer scheme, $n$ a natural number, $\pP \in Z$ a totally split prime, and $X(\Op)_n$ the associated polylogarithmic Chabauty-Kim locus. Suppose $Z$ is not totally real, and assume the period conjecture holds. Then $X(\Op)_n = X(\Op)$.
\end{prop*}

\begin{proof}
In this case, each motivic Ext group $\Ext^1_Z(\QQ(0), \QQ(n))$ is nonzero, and by the period conjecture, each (abelian) syntomic realization map
\[
\Ext^1_Z(\QQ(0), \QQ(n)) \to \Ext^1_{\Op} (\Qp(0), \Qp(n))
\]
is at least nonzero. Since the motivic $\Ext^2$-groups are nevertheless zero, the Selmer scheme
 $H^1(G(Z), U(X)^\PL_{\ge -n})$
is an $\Ext^1_Z(\QQ(0), \QQ(n))$-torsor over $H^1(G(Z), U(X)^\PL_{\ge -(n-1)})$. The realization map
\[
\Re_\pP: H^1(G(Z), U(X)^\PL_{\ge -n}) 
\to H^1(G(\Op), U(X)^\PL_{\ge -n})
\]
is compatible with torsor structures. It follows by induction that $\Re_\pP$ is surjective, which completes the proof of the proposition.
\end{proof}

\segment{18b}{}
Instead of considering the single realization map $\Re_\pP$ we may consider the product, which fits into a square similar to the one considered above
\[
\xymatrix{
X(Z) 
\ar[r] \ar[d]
&
\prod_{\pP | p} X(\Op) 
\ar[d]^-{\al}
\\
 H^1(G(Z), U(X)^\PL_{\ge -n}) 
 \ar[r]_-{\Re_p}
&
\prod_{\pP | p} H^1(G(\Op), U(X)^\PL_{\ge -n}).
}
\]
We define the \emph{big polylogarithmic Chabauty-Kim locus} by
\[
\left( \prod_{\pP | p} X(\Op)  \right)_n
:= \al\inv(\Im \Re_p).
\]
and we again symmetrize with respect to the $S_3$ action to obtain a \emph{symmetrized big polylogarithmic Chabauty-Kim locus}
\[
\left( \prod_{\pP | p} X(\Op)  \right)_n^{S_3}
.
\]
We also write $\Kk_p^\m{Big}(\nN^\PL_{\ge -n})$ and $\Kk_p^\m{Big}(\nN^\PL_{\ge -n})S_3$ for the corresponding ideals of Coleman functions, the \emph{big $p$-adic Chabauty-Kim ideal} and \emph{symmetrized big $p$-adic Chabauty-Kim ideal} respectively. We restrict ourselves as usual to the totally split case for simplicity.

\begin{conj} \label{18c}
\textit{Convergence of polylogarithmic loci, general case. Joint with David Corwin.}
Let $Z$ be an open integer scheme and $p$ a prime below $Z$, for which $Z$ is totally split. For each $n \in \NN$ let $\left( \prod_{\pP | p} X(\Op)  \right)_n^{S_3}$ denote the associated symmetrized big polylogarithmic Chabauty-Kim locus. Then for some $n$ we have 
\[
X(Z) = \left( \prod_{\pP | p} X(\Op)  \right)_n^{S_3}.
\]
\end{conj}

\segment{18d}{}
A straightforward modification of the algorithm $\Aa_\m{Loci}$ yields an algorithm which produces approximate generators for the ideal of polylogarithmic functions defining the big polylogarithmic loci. Its output includes an algebra basis $\Bb$ of $A(\Zo)_\le{n}$ for $\Zo$ an open subscheme of $Z$, as well as a family of elements $\Ftil_i$ of the polynomial ring
\[
\QQ \left[ \Bb, \{ \log t_{\pP} \}_\pP, \{ \Li_i t_{\pP} \}_{i, \pP} \right].
\] 
We denote its output by $\Aa_\m{Big-Loci}(Z,p,n,\ep)$. The result is a theorem which is analogous to the one stated above.

\Theorem{18e}{
Let $Z$ be an open integer scheme, $p$ a prime over which $Z$ it totally split, $n$ a natural number, and $\ep \in p^\ZZ$. 
\begin{enumerate}
\item
Suppose $\Aa_\m{Big-Loci}(Z,p,n,\ep)$ halts. Then there are functions $\{ F^\pP_i\}$ generating the big $p$-adic Chabauty-Kim ideal $\Kk_p^\m{Big}(\nN^\PL_{\ge -n})$ associated to $\nN^\PL_{\ge -n}$ such that 
\[
\left| \widetilde F_i^\pP - F_i^\pP \right| < \ep
\] 
for all $i$.
\item
Suppose Zagier's conjecture (conjecture \ref{int14}) holds for $K$ and $n' \le n$. Suppose Goncharov exhaustion (conjecture \ref{int15}) holds for $Z$ and $n' \le n$. Suppose the period conjecture holds for the open subscheme $\Zo \subset Z$ constructed in segment \ref{6} in half-weights $n' \le n$. Suppose $K$ obeys the Hasse principle for finite cohomology (condition \ref{int19}) in half-weights $2 \le n' \le n$. Then the computation $\Aa_\m{Loci}(Z,p, n, \ep)$ halts.
\end{enumerate}
}

\appendix

\section{A minor erratum}

\segment{}{}
The article \cite{GonMPMTM} contains a minor error: if lemma 3.7 of that article were true, our algorithm could be greatly simplified. However, Cl\'ement Dupont has pointed out the following simple counterexample: $(\log^U 2)\ze^U(3)$ is ramified at $2$. To see this, note that $A(\Spec \ZZ)_4 = 0$, that $A(\Spec \QQ)$ is an integral domain, and that both $\log^U 2$ and $\ze^U(3)$ are nonzero. However, since both of these elements are contained in the space of extensions, in the notation of that article, we have $\Delta'_{[4]}\big((\log^U 2)\ze^U(3)\big) =0$.

\bibliography{MTMUEII_Refs}
\bibliographystyle{alphanum}

\vfill

\noindent
\Small\textsc{Department of mathematics \\ Ben-Gurion University of the Negev\\ Be'er Sheva, Israel}\\ {Email address:} \texttt{ishaidc@gmail.com}

\end{document}